\documentclass[notitlepage]{scrartcl}
\usepackage[shortlabels]{enumitem}	
\usepackage{graphicx}				
\usepackage[export]{adjustbox} 		

\usepackage{amsmath, amsthm, amssymb, amsfonts}
\usepackage{enumitem}
\usepackage{thmtools}
\usepackage{bm}
\usepackage{setspace}
\usepackage{float}
\usepackage{hyperref}
\usepackage{cleveref}
\usepackage{abstract}
\usepackage[utf8]{inputenc}
\usepackage[english]{babel}
\usepackage{framed}
\usepackage[dvipsnames]{xcolor}
\usepackage{tcolorbox}
\usepackage{dsfont}
\usepackage[autostyle]{csquotes}
\usepackage{mathtools}
\usepackage{version}
\usepackage{scalerel}
\excludeversion{confidential}
\hypersetup{
    colorlinks=true,
    linkcolor=blue,
    filecolor=magenta,      
    urlcolor=cyan
    }


\usepackage[firstinits=true, maxbibnames=99, backend=biber, style=numeric, doi=false, url=false, isbn=false]{biblatex}
\usepackage{comment}
\usepackage[margin=1in]{geometry}

\newtheorem{theorem}{Theorem}[section]
\newtheorem{definition}[theorem]{Definition}
\newtheorem{lemma}[theorem]{Lemma}

\newtheorem{proposition}[theorem]{Proposition}
\newtheorem{corollary}[theorem]{Corollary}
\newtheorem{remark}[theorem]{Remark}
\newtheorem{step}{Step}

\newlist{myenumerate}{enumerate}{1}
\setlist[myenumerate,1]{label=\textup{(\roman*)}, labelsep=1em}

\newcommand{\prob}{\ensuremath{\mathbb{P}}}

\newcommand{\longbm}[3]{\ensuremath{\left( \begin{array}{c} {#1} \\ {#2} \\ {#3} \end{array} \right)}} 

\newcommand{\indc}[1]{\ensuremath{1_{#1}}}
\newcommand{\C}{\ensuremath{\mathcal{C}}}
\newcommand{\E}{\ensuremath{\mathbb{E}}}
\newcommand{\F}{\ensuremath{\mathcal{F}}}
\newcommand{\G}{\ensuremath{\mathbb{G}}}
\newcommand{\indcbm}{\ensuremath{(1_{v_{t+1} \in V^{(1)}}, \dots, 1_{v_{t+1} \in V^{(m)}})^\dagger}}
\newcommand{\indcbmi}{\ensuremath{(1_{v_{i+1} \in V^{(1)}}, \dots, 1_{v_{i+1} \in V^{(m)}})^\dagger}}
\newcommand{\R}{\ensuremath{\mathbb{R}}}
\newcommand{\Z}{\ensuremath{\mathbb{R}}}
\newcommand{\hatindcbm}{\ensuremath{(1_{\hat{v}_{t+1} \in V^{(1)}}, \dots, 1_{\hat{v}_{t+1} \in V^{(m)}})^\dagger}}
\newcommand{\hatindcbmi}{\ensuremath{(1_{\hat{v}_{i+1} \in V^{(1)}}, \dots, 1_{\hat{v}_{i+1} \in V^{(m)}})^\dagger}}
\newcommand{\Bin}[2]{\ensuremath{\textrm{Bin}({#1}, {#2})}}
\newcommand{\N}{\ensuremath{\mathbb{N}}}
\newcommand{\footremember}[2]{%
    \footnote{#2}
    \newcounter{#1}
    \setcounter{#1}{\value{footnote}}%
}						

\numberwithin{equation}{section}

\title{\vspace{-1.5cm}The connectivity and phase transition in inhomogeneous random graphs of finite types}
\author{%
Hamin Jung \footremember{trailer}{\scriptsize{Department of Mathematics, Seoul National University. Email: silverbell64@snu.ac.kr}}
}

\begin{document}
\maketitle

\vspace{-1cm}
\begin{abstract}
A significant generalization of the Erd\"os-R\'enyi random graph model is an `inhomogeneous' random graph where the edge probabilities vary according to vertex types. We identify the threshold value for this random graph with a finite number of vertex types to be connected and examine the model's behavior near this threshold value. In particular, we show that the threshold value is $c \frac{\log n }{n}$ for some $c>0$ which is explicitly determined, where $n$ denotes the number of vertices. Furthermore, we prove that near the threshold, the graph consists of a giant component and isolated vertices. We also investigate the phase transition and provide an alternative proof of the results by Bollob\'as et al. [Random Struct. Algorithms, 31, 3-122 (2007)]. Our proofs are based on an exploration process that corresponds to the graph, and instead of relying heavily on branching processes, we employ a random walk constructed from the exploration process. We then apply a large deviations theory to show that a reasonably large component is always significantly larger, a strategy used in both connectivity and phase transition analysis.
\end{abstract}

\section{Introduction}
\label{sec: introduction}

The first random graph model was presented and studied by Erd\"{o}s and R\'{e}nyi \cite{Erdos1959}\cite{Erdos1960} in the late 1950s and early 1960s. They proposed a model, $\G_{n,m}$, with a sample space of graphs consisting of $n$ labeled vertices and $M = M(n)$ edges. The outcomes from the sample space are assumed to be equiprobable. Meanwhile, Gilbert \cite{Gilbert1959} introduced another random graph model, $\G_{n,p}$, which is a graph on $n$ vertices with edges appearing with probability $p = p(n)$, and all the edges are assumed to be independent. In many circumstances, these models $\G_{n,m}$ and $\G_{n,p}$ are equivalent under the obvious relation $p = M / {n\choose2}$. For convenience, the latter is mostly used and referred to as the Erd\"{o}s-R\'{e}nyi random graph. 

One of the most significant properties of random graphs is the phase transition, which is also called as the emergence of the giant component.  For a graph $G$, we write by $\C_1(G)$ the largest component of $G$. We say that an event $E_n$ holds `with high probability', abbreviated as w.h.p., if $\prob(E_n) \to 1$ as $n \to \infty$. The following result is the classical result of the phase transition in the Erd\"os-R\'enyi random graph.
\begin{theorem}
Consider $\G := \G_{n,c/n}$ where $c>0$ is a constant.
\begin{myenumerate}
\item If $c<1$, there exists a $C>0$ such that all components of $\G$ have size less than $C\log n$ w.h.p.
\item If $c>1$, there exist $C,D>0$ such that $Dn \leq |\C_1(\G)| \leq Cn$ and other components have size less than $C \log n$ w.h.p.
\end{myenumerate}
\end{theorem}

Another basic property of a graph is connectivity. Similarly to the giant component, a transition also occurs regarding the connectivity of a graph.

\begin{theorem}
\label{thm: homo connected}
Consider $\G:=\G_{n, c\log n / n}$ where $c>0$ is a constant.
\begin{myenumerate}
\item If $c<1$, $\G$ is not connected w.h.p.
\item If $c>1$, $\G$ is connected w.h.p. 
\end{myenumerate}
\end{theorem}

Note the scale of $p(n)$ for each transition; for the giant component we have $p = 1/n$, and for connectivity we have $p = \log n /n$. 

The Erd\"{o}s-R\'{e}nyi random graph is `homogeneous' in the sense that all the vertices are equivalent. For example, every vertex has the same degree distribution and the same chance of contributing to the largest component. However, this model is not realistic for our real world, since in practical situations, vertices have varying properties. A vertex may be tightly connected to some vertices while having fewer interactions with others. This observation, supported with real experiments and examples, led to the introduction of an `inhomogeneous' random graph model defined as follows.

\begin{definition}
\label{def: graph model 1}
Let $m$ be a positive integer, $\bm{a} = (a_1, \dots, a_m)$ a length $m$ vector with positive entries, and $\bm{c} = (c_{kl})_{1\leq k, l \leq m}$ a $m \times m$ symmetric matrix with non-negative entries. Consider simple graphs (i.e. graphs with no loops or multiple edges) with fixed labeled vertices 
\begin{eqnarray*}
V^{(1)} = \{ 1, \dots , a_1 n \}, \ V^{(2)} = \{ a_1 n +1, \dots , (a_1 + a_2 )n \}, &&\\
 \dots, \ V^{(m)} = \{ (a_1 + \dots + a_{m-1})n + 1 , &\dots& , (a_1 + \dots + a_m )n \}, 
\end{eqnarray*}
where each $V^{(k)}$ denotes the set of vertices of type $k = 1, \dots, m$.  
We denote by $\G (n, m, \bm{a}, \bm{c}, p(n))$ the random graph with vertices $V = V^{(1)} \cup \dots \cup V^{(m)}$ where an edge $xy\ (x \in V^{(k)}, y \in V^{(l)})$ exists with probability $p_{kl} = c_{kl}p$, independently of all other edges.
\end{definition}

Thus $\G (n, m, \bm{a}, \bm{c}, p(n))$ is a random graph with $m$ vertex types, and for each $k = 1, \dots, m$ there are $a_k n$ vertices of type $k$. And as the Erd\"os-R\'enyi model, all the edges are independent.
\begin{figure}[h]
    \centering
    \begin{minipage}[b]{0.3\textwidth}
        \centering
        \includegraphics[width=\textwidth]{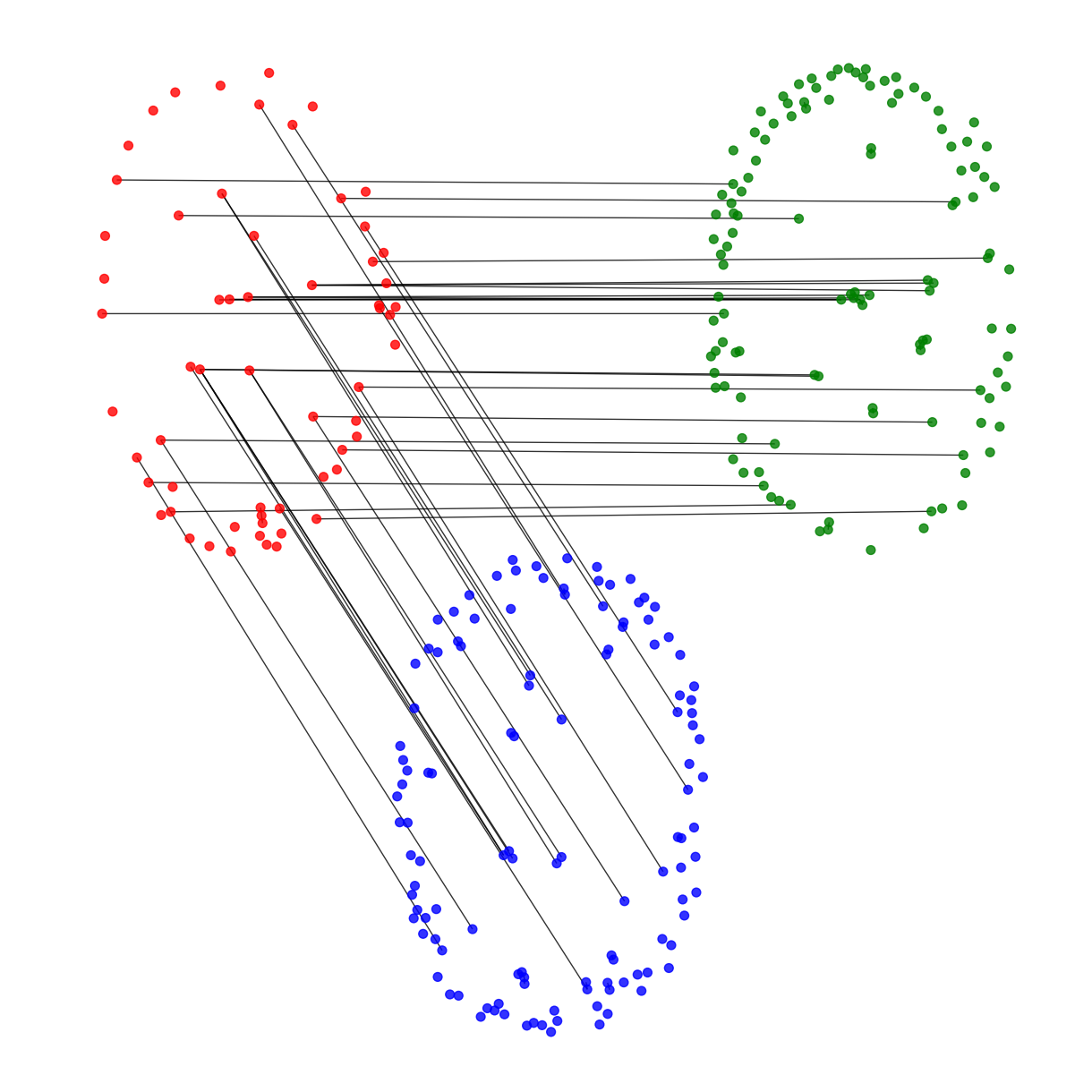}
        \caption{Model with $m=3$, $0.2\bm{c}$}
        \label{fig:image-b}
    \end{minipage}
    \hspace{0.02\textwidth} 
    \begin{minipage}[b]{0.3\textwidth}
        \centering
        \includegraphics[width=\textwidth]{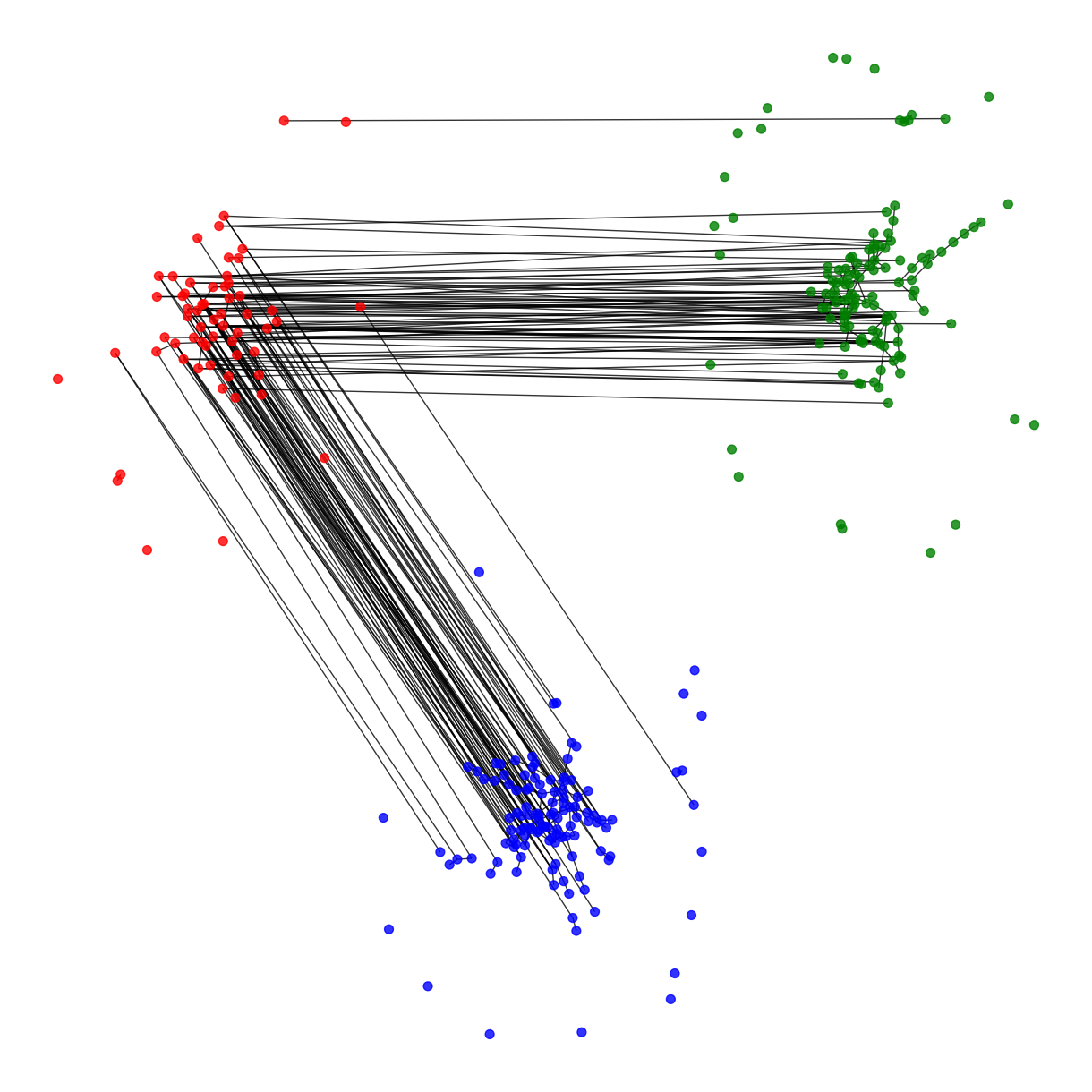}
        \caption{Model with $m=3$, $0.5\bm{c}$}
        \label{fig:image-b}
    \end{minipage}
    \hspace{0.02\textwidth} 
    \begin{minipage}[b]{0.3\textwidth}
        \centering
        \includegraphics[width=\textwidth]{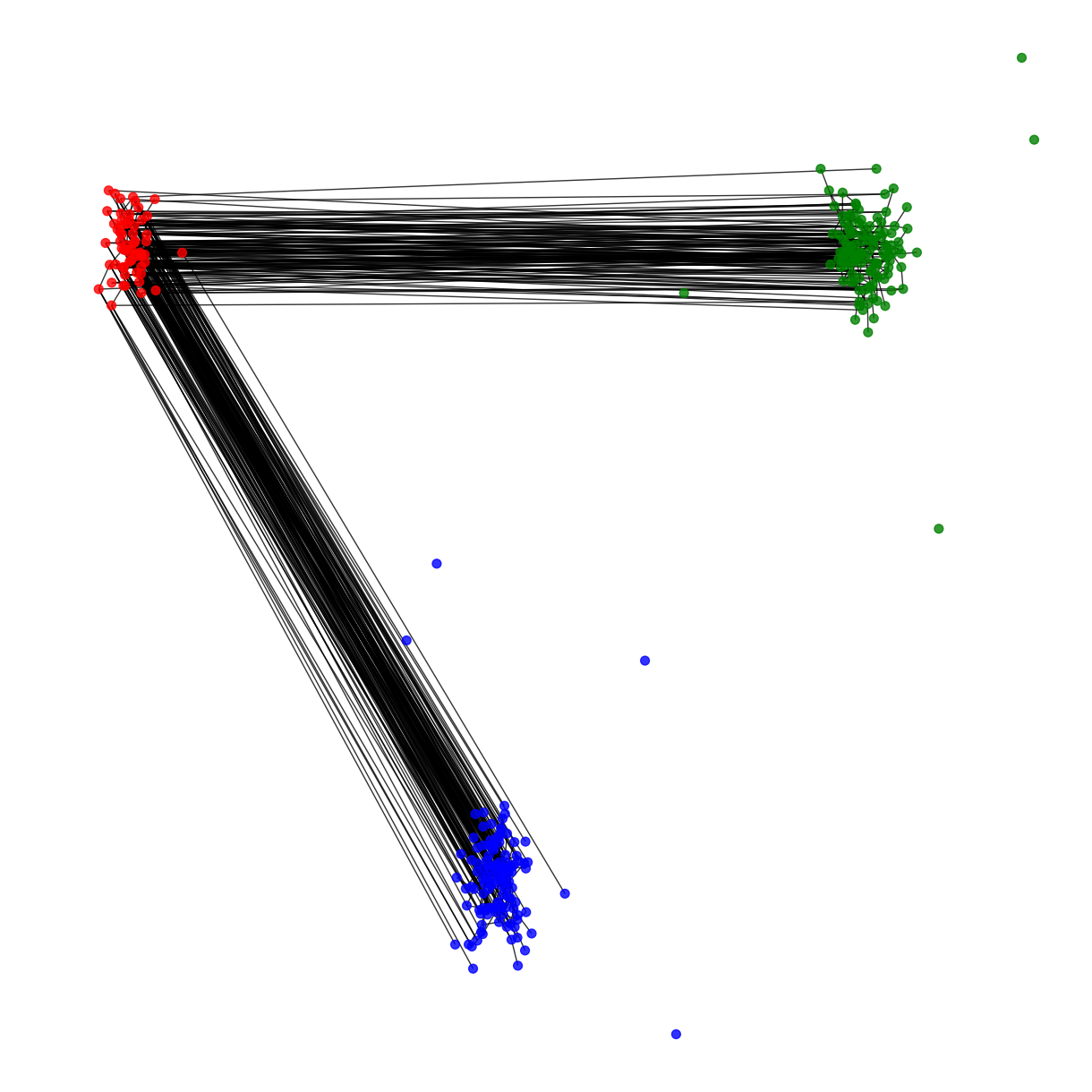}
        \caption{Model with $m=3$, $\bm{c}$}
        \label{fig:image-c}
    \end{minipage}
\end{figure}

There are numerous results on the properties of inhomogeneous random graph models and their variations. S\"{o}derberg \cite{Soderberg2003} stated the emergence of a giant component of inhomogeneous random graph models of finite types. Bollob\'as et al. \cite{Bollobas2007} introduced a highly generalized inhomogeneous model where the vertex types come from a separable metric space equipped with a Borel probability measure. The authors in \cite{Bollobas2007} proved the phase transition by establishing a correspondence between the model and a multi-type Galton-Watson branching process. They first dealt with models of finite types and proceeded to the generalized model by approximation (see Proposition 9.3 of \cite{Bollobas2007}). For connectivity, Alon \cite{Alon1995} showed that a phase transition appears for an even more general inhomogeneous model. Additional properties and references can be found in \cite{Bollobas2007}.

We study the emergence of the giant component and connectivity of the random graph model defined in Definition~\ref{def: graph model 1}. Both involve similar methods, and although our principal result is the transition in connectivity, we first deal with the phase transition because it requires more care in computations. Distinguished from \cite{Soderberg2003}, we present a rigorous and self-contained proof inspired from the method for the homogeneous case in \cite{Durrett2006}. The proof essentially uses random walks as a bound of a process that explores the components, specifically described in Section~\ref{sec: exploration}.

We want to find a threshold for $p(n)$ such that the model $\G = \G (n, m, \bm{a}, \bm{c}, p(n))$ has a giant component. Note that we allow $c_{kl}=0$, and if $c_{kl}=0$ for sufficiently many $k, l \in \{ 1 , \dots , m\}$ the graph $\G$ may be divided into components deterministically. In this case we may divide our random graph and view it as a disjoint union of several random graphs. Hence it is enough to find a threshold value for models satisfying the following assumption:

\begin{remark}
\label{main thm: assumption}
Consider the graph $G'$ obtained from our original random graph model \[\G (n, m, \bm{a}, \bm{c}, p(n))\] by collapsing each set $V^{(k)}$ into one vertex, and two vertices $k, l \in \{ 1, \dots , m \} $ in this collapsed graph $G'$ will be adjacent if and only if $c_{kl} >0$. 
We shall assume that this (deterministic) collapsed graph $G'$ is connected.
\end{remark}
We define a $m \times m$ matrix $\Lambda$ as
\[
\Lambda = (a_l c_{kl})_{1 \leq k, l \leq m}.
\]
Note that $\Lambda_{kl} = a_l c_{kl}$ is the expected number of type $l$ vertices connected to a single type $k$ vertex. 

\begin{definition}
\label{def: associated bp}
Given a random graph model $\G = \G (n, m, \bm{a}, \bm{c}, p(n))$ with $p(n) = \lambda / n$, we associate a multi-type branching process $\bm{Z}_t^{\G}(k)$ of $m$ types of particles, where $(k)$ indicates that the process starts with a single particle of type $k$. A type $l$ particle will have offspring distribution \[ \textrm{Poi}(\lambda a_1 c_{l1} ) \otimes \dots \otimes \textrm{Poi}(\lambda a_m c_{lm}). \] We write $q_k$ as the extinction probability of $\bm{Z}^{\G}_t(k)$.
\end{definition}

See Section~\ref{sec: mbp} for precise definitions of a multi-type branching process and related concepts.

We first write the results shown in \cite{Bollobas2007} (See Example 4.3, Theorems 3.1, 3.12, 9.10).

\begin{theorem}
\label{bollobas thm}
Let $p(n) = \lambda / n$ where $\lambda > 0$ is fixed. Suppose that $\G = \G (n, m, \bm{a}, \bm{c}, p(n))$ satisfies the assumption in Remark~\ref{main thm: assumption}.
\begin{myenumerate}
\item If $\lambda \| \Lambda \| \leq 1$, then $|\C_1(\G)| /n \to 0$ in probability. Furthermore, if $\lambda \| \Lambda \| < 1$, we have $|\C_1(\G)| \leq C\log n$ w.h.p. for some $C>0$.
\item If $\lambda \| \Lambda \| >1$, then $|\C_1(\G)| / \alpha n \to 1$ in probability, where $\alpha = \sum_{k=1}^{m}a_k q_k$ and $q_k$ is the extinction probability from Definition~\ref{def: associated bp}. Moreover, we have 
\[ |\C_1(\G) \cap V^{(k)}| / a_k q_k n \to 1\]
in probability for $k = 1, \dots, m$. 
\item In the setting of case (ii), the second largest component has order less than $C \log n$ for some $C>0$.
\end{myenumerate}
\end{theorem}
Here $\| A \|$ is defined by
\[ \|A\| = \sup \{\|A x \| : \| x \| \leq 1 \}, \]
where $\| x \| ^2 = \sum_{k=1}^{m} a_k x_k ^2 $. 

The number of adjacent vertices of a fixed vertex follows a linear combination of binomial distributions. In \cite{Bollobas2007}, the authors use a standard exploration process (see Section~\ref{sec: exploration}) and squeeze the binomial distributions with Poisson distributions. Then, they couple the exploration process with branching processes associated to these Poisson distributions and use properties of the branching processes to approximate $\prob[x \in B]$, where $B$ is the union of big components. To summarize, they squeeze the union of big components by branching processes, a relatively simple step, and thoroughly use the properties of branching processes. The advantage of this approach, first deriving extensive properties of branching processes and readily verifying graphical properties, is that it is applicable to graphical properties other than the component size.

In this article, we present an alternative proof of the phase transition of the giant component. Instead of heavily using branching processes, we mainly use a random walk and large deviations lemmas. Our result is the following.

\begin{theorem}
\label{main thm: component}
Let $p(n) = \lambda / n$ where $\lambda > 0$ is fixed. Suppose that $\G = \G (n, m, \bm{a}, \bm{c}, p(n))$ satisfies the assumption in Remark~\ref{main thm: assumption}. Let $\mu$ be the largest eigenvalue of $\Lambda$ (See Remark~\ref{rmk: mbp}). 
\begin{myenumerate}
\item If $\lambda \mu < 1$, there exists a $C>0$ such that $|\C_1(\G)| \leq C\log n$ w.h.p. 
\item If $\lambda \mu >1$, then $|\C_1(\G)| / \alpha n \to 1$ in probability, where $\alpha = \sum_{k=1}^{m}a_k q_k$. Moreover, we have 
\[ |\C_1(\G) \cap V^{(k)}| / a_k q_k n \to 1 \]
in probability and $q_k <1$ for all $k=1,\dots,m$. In addition, all the other components have size~${\leq \beta \log n}$ w.h.p., where $\beta > 0$ is a constant completely determined by $m, \Lambda$ and $\lambda$.
\end{myenumerate}
\end{theorem}

\begin{remark}
We note that Theorem~\ref{bollobas thm} and Theorem~\ref{main thm: component} assert the same critical value for $\lambda$ since $\| \Lambda \| = \mu$. This is easily seen by observing that $\Lambda$ is a self adjoint operator with respect to the space $V = \R^m$ equipped with the inner product $\langle x,y\rangle = \sum_{k=1}^m a_k x_k y_k$, and applying the spectral theorem.
\end{remark}


The next result is on the threshold for connectivity. As our previous illustration, if $c_{kl} = 0$ for many $k, l \in \{ 1 , \dots , m\}$ the graph is disconnected with probability 1. Hence we shall again assume Remark~\ref{main thm: assumption}.

We first state the result of Alon \cite{Alon1995} which is applicable to more general inhomogeneous models.

\begin{theorem}
\label{alon thm}
Let $G = (V, E)$ be a simple graph with probability $p_e \in [0,1]$ assigned to each edge $e \in E$. Denote by $G_p$ the random graph obtained from $G$ by deleting each edge $e \in E$ with probability $1-p_e$, independently and randomly. \\
Then, for every $b>0$, there exists a constant $c = c(b) >0$ such that if for every nontrivial $S \subset V$
\[ \sum_{x \in S, y \in V \setminus S} p_{xy} \geq c \log n, \]
then the probability that $G_p$ is connected is at least $1- n^{-b}$.
\end{theorem}

Theorem~\ref{alon thm} is consistent with the case of homogeneous models in the sense that it has the same scale for $p$; it implies that $ p = \alpha \log n /n$ for some $\alpha>0$ gives connectedness. However, Theorem~\ref{alon thm} does not give the exact threshold value for connectivity nor additional information on the critical regime. We shall use similar methods as in our analysis of the giant component to find the exact critical value $p$ for connectedness.

\begin{theorem}
\label{main thm: connectivity 1}
Let $p(n) = \frac{\alpha \log n}{n}$. Define $b_k = \sum_{l=1}^{m} a_l c_{kl}$ for $k = 1, \dots , m$ and let \[b = \min \{b_1, \dots, b_m\}.\] Suppose that $\G = \G (n, m, \bm{a}, \bm{c}, p(n))$ satisfies the assumption in Remark~\ref{main thm: assumption}. 
\begin{myenumerate}
\item If $\alpha b <1$, then $\G$ is not connected w.h.p.
\item If $\alpha b > 1$, then $\G$ is connected w.h.p.
\end{myenumerate}
\end{theorem}

We also investigate the critical case $p(n) = (\log n + c + o(1))/ bn$.

\begin{theorem}
\label{main thm: critical connectivity}
Let $p(n) = (\log n + c + o(1))/ bn$ and $b$ and $\G$ is defined as in Theorem~\ref{main thm: connectivity 1}. Assume Remark~\ref{main thm: assumption}. Then $\G$ consists of a giant component and some isolated vertices w.h.p. Moreover, the number of isolated vertices in $\G$ converges to a Poisson distribution.
\end{theorem}

Although it is possible that $p_{kl} = 0$ in our model $\G (n, m, \bm{a}, \bm{c}, p(n))$, the probabilities $p_{kl}$ have the same scale $p(n)$. We would like to consider another random graph model where 
\[p_{k'l'} = c_{k'l'} p'(n) \textrm{ for some } k', l', \ p' (n) = o(p(n)),\]
in other words, some vertices are much less likely to be adjacent. We therefore construct a random model by taking the disjoint union of $\G (n, m_i, \bm{a}_i, \bm{c}_i, p(n))$ for $i = 1, \dots , r$. We call each $\G (n, m, \bm{a}_i, \bm{c}_i, p(n))$ as `block $G_i$'. The probability of two vertices from distinct blocks $G_i$ and $G_j$ being adjacent is $q_{(i,k)(j,l)} = d_{(i,k)(j,l)} p'(n)$ where $k, l$ are the vertex types. We summarize as:

\begin{definition}
\label{def: graph model 2}
For each $i = 1, \dots, r$, let $m_i$ be a positive integer, $\bm{a}_i$ a length $m_i$ vector with positive entries, and $\bm{c}$ a $m_i \times m_i$ symmetric matrix with non-negative entries. Write $M_i = \sum_{j=1}^{i} m_i$. Let $\bm{d}$ be a $M_r \times M_r$ symmetric matrix with non-negative entries. \\
Define \[\G (n, r, (m_i), (\bm{a}_i), (\bm{c}_i), \bm{d}, p(n), p'(n))\]
as the random graph formed by taking the disjoint union of 
\[G_i := \G (n, m_i, \bm{a}_i, \bm{c}_i, p(n)),\  i = 1, \dots, r \]
and adding edges $x \in G_i, y \in G_j$ for $i \neq j$ with probability \[q_{(i,k)(j,l)} = d_{(M_{i-1}+k)(M_{j-1}+l)} p'(n),\] independently, where $k, l$ are the vertex types of $x, y$, respectively.
\end{definition}

We again assume Remark~\ref{main thm: assumption} for each $G_i$. We additionally assume the following, an analog to Remark~\ref{main thm: assumption}.

\begin{remark}
\label{main thm: assumption 2}
Consider the graph $G''$ obtained from our original random graph model \[\G (n, r, (m_i), (\bm{a}_i), (\bm{c}_i), \bm{d}, p(n), p'(n))\] by collapsing each block $G_i$ into one vertex, and define two vertices $i, j \in \{1 , \dots , r\}$ to be adjacent if and only if \[d_{(M_{i-1}+k)(M_{j-1}+l)} > 0\] for some type $k$ from block $G_i$ and some type $l$ from block $G_j$. We shall assume that $G''$ is connected.
\end{remark}

From Alon \cite{Alon1995} we may anticipate that $\bm{d}$ has no effect on the threshold value, which is true according to the following theorem.

\begin{theorem}
\label{main thm: connectivity 2}
Assume Remarks~\ref{main thm: assumption} and \ref{main thm: assumption 2}. Let $p(n) = \frac{\alpha \log n}{n}$, $p'(n) = o(p(n))$ and define $b_k ^{(i)}$ as in Theorem~\ref{main thm: connectivity 1} for each block $G_i$. \\
(i) If $\alpha b_k ^{(i)} < 1$ for some $i, k$ or $n^2 p'(n)= o(1)$, then $G$ is not connected w.h.p. \\
(ii) If $\alpha b_k ^{(i)} > 1$ for all $i, k$ and $n^2 p'(n) \to \infty$, then $G$ is connected w.h.p.
\end{theorem}

Theorem~\ref{main thm: connectivity 2} accounts for networks where some groups are sparsely connected.


\section{Preliminaries}

\subsection{Notations}
\label{notations}
We introduce some notations that are frequently used throughout this article. We first present standard notations for asymptotics, which are found in \cite[Chapter 1]{Janson2000}. We are interested in the relative order of two sequences $(a_n), (b_n)$ of non-negative real numbers.
\begin{itemize}
	\item $a_n = O(b_n)$ as $n \to \infty$ if $\exists C, n_0$ such that $a_n \leq C b_n$ for $n \geq n_0$.
	\item $a_n = \Theta(b_n)$ as $n \to \infty$ if $\exists C>0, c>0, n_0$ such that $cb_n \leq a_n \leq C b_n$ for $n \geq n_0$.
	\item $a_n = o(b_n)$ as $n \to \infty$ if $a_n / b_n \to 0$ as $n \to \infty$.
	\item $a_n \sim b_n$ if $a_n /b_n \to 1$.
\end{itemize}
We shall omit the phrase ``as $n \to \infty$" if it is clear from context.

Next, we give asymptotic notations involving probability space. Let $X_n$ be a random variable and $(a_n)$ a sequence of non-negative real numbers.
\begin{itemize}
	\item $X_n = O_C(a_n)$ as $n \to \infty$ if $\exists C>0$ such that $|X_n| \leq Ca_n$ w.h.p.
	\item $X_n = \Omega_C (a_n)$ as $n \to \infty$ if $\exists c>0$ such that $X_n \geq c a_n$ w.h.p.
	\item $X_n \sim a_n$ as $n \to \infty$ if $X_n / a_n \to 1$ in probability.
\end{itemize}
As before we may omit ``as $n \to \infty$" when there is no ambiguity. Also, since we are mostly interested in the behavior when $n \to \infty$, we may proceed with inequalities or equations which only hold for sufficiently large $n$ \textit{without} including the phrase ``for sufficiently large $n$".

We denote as Ber($p$) the distribution of a random variable $X$ satisfying \[\prob[X=0]=1-p, \ \prob[X=1]=p,\] i.e. the Bernoulli distribution with parameter $p$. We also denote as Poi($\lambda$) as the Poisson distribution with mean $\lambda$.  And finally Bin$(n, p)$ is the binomial distribution with parameters $n, p$. As an abuse of notation, we may use these distributions to denote random variables following that particular distribution; for example, in equations or inequalities, Bin$(n,p)$ might appear and indicate a random variable following the distribution Bin$(n,p)$. And given a probability distribution $F$ and a random variable $X$, we write $X \sim F$ to indicate that $X$ has as distribution $F$.

In each argument we have a specific random graph model $\G$. We write the probability space for $\G$ as $(\Omega, \F, \prob)$. We shall denote as $G$ a graph chosen randomly from the model $\G$. We sometimes add a subscript to $\prob$ to denote the starting point of a process as in Markov chains. Probabilities involving graph properties will never have subscripts, e.g. $\prob (G$ is connected), but probabilities involving a process (that we will define shortly) with varying initial states will have subscripts when we are calculating under specific initial conditions; they will not have subscripts if the calculations hold for any initial state.

We now state some elementary notations. We shall write $\bm{x}=(x_1, x_2, \dots, x_d)$ to denote a row vector in $\R^d$, and $\bm{y} = (y_1, y_2, \dots, y_d)^{\dagger}$ to denote a column vector in $\R^d$. Note that we use the same notation for row and column vectors when represented by $\bm{x}$, but distinguish them by putting a transpose superscript when they are represented by $n$-tuples. We shall write $\bm{x} \geq \bm{y}$ or $\bm{y} \leq \bm{x}$ if all the components of $\bm{x}$ are greater or equal to the corresponding components of $\bm{y}$. $\bm{e}_k$ shall be the Euclidean vector whose $k$-th component is 1 and others are 0. And for the inner product of $\bm{x}, \bm{y} \in \R^d$ we shall write $\bm{x} \bm{y}^\dagger$.

Recall that for a graph $G=(V,E)$ we wrote $\C_1(G)$ for the largest component of $G$. For $x \in V$ we shall write $\C_x$ to denote the component of $G$ containing $x$.


\subsection{Multi-type branching processes}bm
\label{sec: mbp}
We state some well-known facts about multi-type Galton-Watson branching processes. Most of the results are from \cite[Chapter 2]{Harris1963}, and we mostly follow their notations.

\begin{definition}
Given an integer $m>0$ and distributions $F_k\ (k=1,\dots,m)$ with state space $\Z_{\geq 0}^m$, a `multi-type branching process' is a vector Markov process $\bm{Z}_0, \bm{Z}_1, \dots$ where each $\bm{Z}_t$ is a length $m$ vector with non-negative integer components. $\bm{Z}_0$ is called the \textit{initial state}, and given $\bm{Z}_t=(x_1, \dots, x_m)$, $\bm{Z}_{t+1}$ is defined by
\[ \bm{Z}_{t+1} = \sum_{k=1}^{m}\sum_{i=1}^{x_k} \bm{\eta}_i ^{(k)}\]
where $\bm{\eta}_i^{(k)}$ are independent random variables with $\bm{\eta}_i^{(k)} \sim F_k.$
\end{definition}
As usual, the particles in $\bm{Z}_{t+1}$ are thought as the `offsprings' of the particles in $\bm{Z}_t$. In this article, the number of types $m$ and each distribution $F_k$ will be clear from the context and we will not write them explicitly. When the initial state is $\bm{Z}_0 = (x_1, \dots, x_m)$, we write subscripts in the probability and expectation to indicate the fixed initial state. For example, we write $\prob_{(x_1, \dots, x_m)}$ or $\prob_A$ where $A$ is a set consisting of $x_k$ elements of type $k$ particles, $k=1,\dots,m$. When $\bm{Z}_0 = \bm{e}_k$ for some $k$, we usually write $\prob_x$ where $x$ is any particle of type $k$, instead of $\prob_{\{x\}}$ or $\prob_{\bm{e}_k}$.
\\

In a 1-type branching process, the expected value of the offspring distribution was the critical value for extinction. In a $m$-type branching process, we use a $m \times m$ matrix $\mathbf{M}$ called the `associated matrix', whose $(k, l)$-th entry is the expected value of the number of type $l$ offsprings from a single particle of type $k$. We write $q_k$ as the extinction probability of the branching process that starts from one individual of type $k$, i.e. $\bm{Z}_0 = \bm{e}_k$. 

\begin{remark}
\label{rmk: mbp}
The Perron-Frobenius theorem asserts that given a square matrix $A$ with non-negative entries, there exists $\mu \geq 0$, called the `Perron root', such that for every other eigenvalue $\lambda$ of $A$ we have $|\lambda|\leq \mu$. It also guarantees the existence of an eigenvector corresponding to $\mu$ with non-negative components. For our associated matrix $\mathbf{M}$ we write the Perron root as $\mu \geq 0$. 
\end{remark}

We shall exclude branching processes where every individual has exactly one offspring. These trivial processes are called `singular', and we only consider non-singular cases. We call $\mathbf{M}$ a `positively regular' matrix if all the entries of $\mathbf{M}$ are finite, and there exists a $n \in \mathbb{N}$ such that all the entries of $\mathbf{M}^n$ are positive. 

\begin{theorem}[Theorem 7.1, \cite{Harris1963}]
\label{thm: mbp}
Consider a non-singular multi-type branching processes such that $\mathbf{M}$ is positively regular. Then the Perron root of $\mathbf{M}$ satisfies $\mu>0$. In the case of $\mu>1,$ we have $q_k <1$ for all $k = 1, \dots, m$.
\end{theorem}

Recall the multi-type branching process $\bm{Z}^{\G}_t(k)$ defined in Definition~\ref{def: associated bp}. Under the assumption of Remark~\ref{main thm: assumption}, if additionally $c_{kk}>0$ for some $k$, we may apply Theorem~\ref{thm: mbp}, which we summarize as:

\begin{corollary}
\label{cor: mbp}
Suppose that the assumption in Remark~\ref{main thm: assumption} holds and $\exists k$ such that $c_{kk}>0$. Assume that the associated matrix of $\bm{Z}^{\G}_t$, $\lambda \Lambda$, has as Perron root $\mu$ satisfying $\lambda \mu >1$. Then the extinction probabilities $q_k$ satisfy $q_k <1$ for all $k = 1, \dots, m$.
\end{corollary}

The extinction probabilities in Corollary~\ref{cor: mbp} later appear in the asymptotic size of the giant component. Although we only showed the conclusion $q_k < 1$ in Corollary~\ref{cor: mbp} for random graph models satisfying $c_{kk}>0\ \exists k$, it will later turn out to be true for models only satisfying Remark~\ref{main thm: assumption}, thus proving the assertion $q_k <1$ in Theorem~\ref{main thm: component}.


\subsection{The exploration process}
\label{sec: exploration}
We introduce two exploration processes on the vertices of $\G (n, m, \bm{a}, \bm{c}, p(n))$ which inspects one vertex at a time. The processes are defined in \cite{Durrett2006} for the homogeneous case, and we present natural extensions to the inhomoegeneous case. Recall
\begin{eqnarray*}
V^{(1)} = \{ 1, \dots , a_1 n \}, \ V^{(2)} = \{ a_1 n +1, \dots , (a_1 + a_2 )n \}, &&\\
 \dots, \ V^{(m)} = \{ (a_1 + \dots + a_{m-1})n + 1 , &\dots& , (a_1 + \dots + a_m )n \}, 
\end{eqnarray*}
where each $V^{(k)}$ denotes the set of vertices of type $k$. $V = V^{(1)} \cup \dots \cup V^{(m)}$ is the set of all vertices. For $x, y \in V$, $\eta_{x,y}$ denotes the random variable which has value 1 if $x, y$ are adjacent and 0 otherwise; note that $\eta_{x,y} = \eta_{y, x}$ and $\eta_{x, y}$ are independent for $1 \leq x < y \leq |V|$.

In our exploration processes, each vertex will be marked as `removed' or `active' or `unexplored'. At time $t \in \Z_{\geq 0}$, $R_t$ will denote the set of removed vertices, $A_t$ the active, and $U_t$ the unexplored. The starting point $A_0$ will be an arbitrary set of vertices, which may differ according to each argument. We also define \[R_t ^{(k)} = R_t \cap V^{(k)},\ k=1,\dots, m,\] and $A_t ^{(k)}, U_t ^{(k)}$ analogously. Let 
\[ \bm{R}_t = (|R_t ^{(1)}|, \dots,|R_t ^{(m)}|)^\dagger \]
and define \( \bm{A}_t, \bm{U}_t \)  analogously.

\subsubsection{The upper exploration process}
\label{sec: upper exploration}
We first define the `upper' exploration process that is used as an upper bound for the component size. Let \[ \tau = \textrm{inf}\{ t : A_t = \emptyset \}, \] and for $t < \tau$, select a vertex $v_{t+1} = \textrm{min}A_t$ for time $t+1$; put $v_{t+1} = \textrm{min} (A_{\tau} \cup U_{\tau})$ for $t \geq \tau$. The sets $R_{t+1}, A_{t+1}, U_{t+1}$ are defined by the following.

\begin{itemize}
	\item For $t=0$, \( R_0 = \emptyset , \ U_0 = V \setminus A_0 .\)
	\item For $0<t<\tau $, \begin{eqnarray}
		R_{t+1} &=& R_{t} \cup \{ v_{t+1} \}, \nonumber \\
		A_{t+1} &=& (A_{t} \setminus \{ v_{t+1} \})  \cup \{ y \in U_t : y, v_{t+1} \textrm{ are adjacent}\}, \label{eq: exploration} \\
		U_{t+1} &=& U_{t} \setminus \{ y \in U_t :  y, v_{t+1} \textrm{ are adjacent} \}. \nonumber
		\end{eqnarray}
	\item For $t \geq \tau$, \( R_t = R_\tau, \ A_t = A_\tau, \ U_t = U_\tau. \)
\end{itemize}
Note that $V = R_t \cup A_t \cup U_t$ and $R_t, A_t, U_t$ are disjoint for all $t$. 
With this exploration process we define a random walk $\bm{S}_t$ by the following.
\begin{itemize}
	\item For $t=0$, \( \bm{S}_0 = \bm{A}_0.\) 
	\item For $0<t<\tau $, \begin{eqnarray} 
		\bm{S}_{t+1} = \bm{S}_t &-& (1_{v_{t+1} \in V^{(1)}}, \dots, 1_{v_{t+1} \in V^{(m)}})^\dagger 	\nonumber \\
		&+& \Biggl(\ \ \displaystyle \smashoperator{\sum_{\scriptscriptstyle y \in U_t ^{(1)}}} \eta_{y, v_{t+1}} + \smashoperator{\sum_{\scriptscriptstyle y \in V^{(1)} \setminus U_t ^{(1)}}} \zeta_y ^{t,(1)}, \dots, \smashoperator{\sum_{\scriptscriptstyle y \in U_t ^{(m)}}} \eta_{y, v_{t+1}} + \smashoperator{\sum_{\scriptscriptstyle y \in V^{(m)} \setminus U_t ^{(m)}}} \zeta_y ^{t,(m)} \Biggr)^\dagger. \label{def: upper random walk}
		\end{eqnarray}
		
Here  \(\zeta_y ^{t,(l)} \sim \textrm {Ber(}p_{kl})\) for all $y \in V^{(l)} \setminus U_t ^{(l)}$, where $k$ is the vertex type of $v_{t+1}$, hence has the same distribution with $\eta_{y, v_{t+1}}$ for each $y \in U_t ^{(l)}$. Of course all the random variables \[ \eta_{y, v_{t+1}}\textrm{ and } \zeta_y^{t, (l)} \ \  (y \in V,\ l = 1, \dots , m,\  t \geq 1) \] are independent.
	\item For $t \geq \tau$,
	\[\bm{S}_{t+1} = \bm{S}_t - (1_{v_{t+1} \in V^{(1)}}, \dots, 1_{v_{t+1} \in V^{(m)}})^\dagger 	
		+ \Biggl(\displaystyle \smashoperator{\sum_{\scriptscriptstyle y=1}^{\scriptscriptstyle a_1 n}} \zeta_y ^{t,(1)}, \dots,  \smashoperator{\sum_{\scriptscriptstyle y=1}^{\scriptscriptstyle a_m n}} \zeta_y ^{t,(m)} \Biggr)^\dagger,  \]
		where  \(\zeta_y ^{t,(l)} \sim \textrm {Ber(}p_{kl})\) are independent variables for all $y, \ t $.
\end{itemize}

\begin{remark}
Unlike \cite{Durrett2006}, we continue to choose $v_{t+1} = \textrm{min} (A_{\tau} \cup U_{\tau})$ for $t \geq \tau$. This is not a significant change as it does not change the exploration process. We extend the selection of $v_{t+1}$ only to ensure that $S_t$ is well defined for all $t$ with increments of consistent distribution. This later allows us to construct a martingale defined for all times.
\end{remark}


\subsubsection{The lower exploration process}
\label{subsec: new exploration}

Our second exploration process works as a lower bound in future arguments. To distinguish from the upper exploration process, we add a superscript $\hat{ }$ to corresponding notions, for example, $\hat{R}_t, \hat{A}_t, \hat{U}_t $ and $\hat{v}_{t+1} $. Let $\delta >0$. We shall inspect one vertex at a time by choosing $\hat{v}_{t+1} = \min \hat{A}_t$ like the upper process.

However, when exploring and activating the vertices in $\hat{U}_t$, we shall only use a portion of $\hat{U}_t$. To be precise, we shall take the smallest $(1- \delta )a_k n$ elements from each $\hat{U}_t ^{(k)}$ to form the set $\hat{U}_{t, \delta} ^{(k)}$ for $k = 1, \dots , m$. In order to make this valid, we will run the process only until $\hat{A}_t \neq \emptyset $ and $| \hat{U}_t ^{(k)} | \geq (1- \delta ) a_k n$ for all $k = 1, \dots , m$. Thus the process runs until the stopping time 
\[
T_W := \textrm{inf} \{ t : \hat{A}_t = \emptyset \textrm{ or } \hat{\bm{A}}_t + \sum_{i=0}^{t-1} \hatindcbmi \ngeq \delta n \bm{a}^\dagger \}.
\] 
Put $\hat{U}_{t, \delta } = \hat{U}_{t, \delta } ^{(1)} \cup \dots \cup \hat{U}_{t, \delta} ^{(m)}$. Now given $\hat{A}_0$, we define the process by the following.

\begin{itemize}
	\item For $t=0$,  \( \hat{R}_0 = \emptyset ,\  \hat{U}_0 = V \setminus \hat{A}_0 .\)
	\item For $0<t<T_W$, \[ \begin{array}{l} 
		\hat{R}_{t+1} = \hat{R}_{t} \cup \{ \hat{v}_{t+1} \}, \\
		\hat{A}_{t+1} = (\hat{A}_{t} \setminus \{ \hat{v}_{t+1} \})  \cup \{ y \in \hat{U}_{t, \delta} : y, \hat{v}_{t+1} \textrm{ are adjacent}\}, \\
		\hat{U}_{t+1} = \hat{U}_{t} \setminus \{ y \in \hat{U}_{t, \delta} : y, \hat{v}_{t+1} \textrm{ are adjacent}\}.
	\end{array} \]
	\item For $t \geq T_W$, \( \hat{R}_t = \hat{R}_{T_W}, \ \hat{A}_t = \hat{A}_{T_W}, \ \hat{U}_t = \hat{U}_{T_W} .\)
\end{itemize}

With this exploration process we define a random walk $\bm{W}_t$ by the following.
\begin{itemize}
	\item For $t=0$, \( \bm{W}_0 = \hat{\bm{A}}_0\). 
	\item For $0<t<T_W $, \begin{equation} 
		\displaystyle \bm{W}_{t+1} = \bm{W}_t - \hatindcbm +\Biggl(\ \smash{\sum_{\scriptscriptstyle y \in \hat{U}_{t,\delta}^{(1)}} \eta_{y, \hat{v}_{t+1}}}, \dots, \sum_{\scriptscriptstyle y \in \hat{U}_{t,\delta} ^{(m)}} \eta_{y, \hat{v}_{t+1}}\Biggr)^\dagger \label{def: lower random walk}
		\end{equation}
	\item For $t \geq T_W$, \begin{equation}
	\label{eq: after T_W}
		\displaystyle \bm{W}_{t+1} = \bm{W}_t - \hatindcbm + \Biggl(\ \smash{\sum_{\scriptscriptstyle y=1}^{\scriptscriptstyle (1-\delta)a_1 n} \zeta_{y}^{t, (1)}}, \dots, \smash{\sum_{\scriptscriptstyle y=1}^{\scriptscriptstyle (1-\delta)a_m n} \zeta_{y}^{t, (m)}}\Biggr)^\dagger,
	\end{equation}
	where $\hat{v}_{t+1} = \min(\hat{A}_\tau \cup \hat{U}_\tau)$. In (\ref{eq: after T_W}), $\zeta_{y}^{t,(l)}$ $\sim$ Ber$(p_{kl})$ for $k$ representing the type of vertex $\hat{v}_{t+1}$. As before all of the random variables \[\eta_{y, \hat{v}_{t+1}} (y \in \hat{U}_{t, \delta},\  t < T_W) \textrm{ and } \zeta_{y}^{t,(l)} (1 \leq y \leq (1-\delta)a_l n, \ l = 1, \dots , m, \ t \geq T_W) \] are independent.
\end{itemize}	

We note that if the upper and lower exploration processes start with the same set of activate vertices, i.e. $A_0 = \hat{A}_0$, then $\hat{v}_{t+1} = v_{t+1}$ and $\hat{\bm{A}}_{t+1} \leq \bm{A}_{t+1}$ for $t < T_W$. Also $T_W \leq \tau$ and
\begin{equation}
\hat{\bm{A}}_t = \bm{W}_t \ \ \ \forall t \leq T_W. 	\label{eq: A=W}
\end{equation}


\section{The largest component}



We mainly follow the methods given in Sections 2.2 and 2.3 of \cite{Durrett2006}. 

In this section, we shall always assume Remark~\ref{main thm: assumption}. Recall that if this assumption does not hold, we may divide the graph $G'$ into components and divide $V$ into sections that have no effect on each other in the exploring process.

Fix a random graph model $\G (n, m, \bm{a}, \bm{c}, p(n))$. Recall the matrix $\Lambda$ defined in Section~\ref{sec: introduction} and Remark~\ref{rmk: mbp}. Consider the row vector $\bm{u}=(u_1, \dots, u_m)$ whose transpose is the eigenvector corresponding to the Perron root $\mu$, i.e.
\begin{equation}
\label{eigenvector}
\Lambda \bm{u}^\dagger = \mu \bm{u}^\dagger \textrm{ or } \bm{u} \Lambda^\dagger = \mu \bm{u}.
\end{equation}
Remark~\ref{rmk: mbp} allows us to assume $0 \leq u_i \leq 1$ for all $i = 1, \dots , m$. Note that in the homogeneous case, we may consider the eigenvector as $u = 1>0$. The Perron-Frobenius theorem only assures that the components of the eigenvector $\bm{u}$ are non-negative, but to apply the methods in \cite{Durrett2006}, we need that all the components are positive. Fortunately, thanks to our assumption in Remark~\ref{main thm: assumption}, we have the following.

\begin{proposition}
\label{prop: u>0}
We have $u_k >0$ for all $k=1, \dots , m$.
\end{proposition}

\begin{proof}
Suppose without loss of generality that $u_1 = \dots = u_l = 0$, $u_{l+1} = \dots = u_m > 0$, where $1 \leq l <m$.
$u_1 = 0$ implies 
\[ u_1 a_1 c_{11} + u_2 a_2 c_{12} + \dots + u_m a_m c_{1m} = 0, \] 
and since every term in the L.H.S. is non-negative, $a_k > 0$ for all $k$, we have
\[ u_1 c_{11} = u_2 c_{12} = \dots = u_m c_{1m} = 0. \]
Proceeding with $u_2 = 0, \dots, u_l = 0$ gives
\begin{eqnarray*}
0 & = &u_1 c_{11} = u_2 c_{21} = \dots = u_m c_{m1} \\
   & = &u_1 c_{12} = u_2 c_{22} = \dots = u_m c_{m2} \\
   & = & \cdots \  = u_1 c_{1l} = u_2 c_{2l} = \dots = u_m c_{ml}.
\end{eqnarray*}
Hence we have $c_{k k'} = 0$ for $k=1, \dots, l,\ k' = l+1, \dots, m$. This implies that the collapsed graph $G'$ may be divided into sections $\{ 1, \dots , l\}$ and $\{ l+1, \dots , m\}$, which is a contradiction to our assumption in Remark~\ref{main thm: assumption}. Hence we must have $u_k > 0$ for all $k$.
\end{proof}

Denote $u = \min\{u_1 , \dots , u_m \}$ and $U=\max\{u_1, \dots, u_m\}$. We also define a filtration $(\mathcal{F}_t)_{t \geq 0}$ of $\sigma$-algebras where each $\mathcal{F}_t$ is the $\sigma$-algebra generated by all sets occuring in the upper exploration process until time $t$. Then we have $\indc{v_{t+1} \in V^{(k)} } \in \F _t$ for all $k=1,\dots , m$. This gives
\[ 
\E [ \bm{S}_{t+1} - \bm{S}_t | \F _t ] = (np\Lambda^\dagger - I )\indcbm 
\]
and
\begin{equation}
\E [\bm{u} (\bm{S}_{t+1} - \bm{S}_t) | \F _t ] = (np\mu -1) \bm{u} \indcbm.	\label{uS increment}
\end{equation}
Now define
\[
X_0 = \bm{u}\bm{S}_0,\ X_t = \bm{u} \bm{S}_t - \sum_{i=0}^{t-1} (np\mu -1) \bm{u} \indcbmi .
\]
Then (\ref{uS increment}) gives \( \E [X_{t+1} - X_t | \F _t ]= 0\), hence $(X_t)_{t \geq 0}$ is a $(\F _t)_{t \geq 0}$-martingale.

Define a stopping time \( T = \mathrm{inf}\{ t: \bm{S}_t = \bm{0} \} \). Since $\bm{S}_t \geq \bm{A}_t$, we have $T \geq \tau$. Hence the random walk $\bm{S}_t$ will give an upper bound of our component size. From now on in this section, we shall use $p(n) = \lambda /n$ where $\lambda>0$ is a constant.


\subsection{The subcritical case: $\mu \lambda<1$}

Assume $\mu \lambda< 1$ throughout this subsection. Our goal here is to prove part (i) of Theorem~\ref{main thm: component}. Let the upper exploration process start from a vertex $x \in V^{(k)}$. We write $\prob_x, \E_x$ to indicate that the process starts from $A_0 = \{x\}$. Then the optional stopping theorem gives
\begin{eqnarray*}
u_k = \E_x X_0 &=& \E_x X_{t \wedge T} \\
		&=& \E_x \left[ \bm{u} \bm{S}_{t \wedge T}^\dagger - \sum_{i=0}^{t \wedge T -1} (\mu \lambda-1) \bm{u} \indcbmi \right] \\
		&\geq& (1-\mu \lambda) u \E_x [t \wedge T].
\end{eqnarray*}
Letting $t \to \infty $, we have
\[
\E_x T \leq \frac{u_k}{(1-\mu \lambda) u} < \infty.
\]
In particular, $T < \infty \ \prob_x\textrm{-a.s.}$
By using the moment generating function as in \cite{Durrett2006}, we may prove that the largest component of our graph $G$ has size $O_C(\log n)$. But we first state a simple observation that is needed to incorporate (\ref{eigenvector}) into our calculations. Given a constant $0 < a \leq 1$, we have
\begin{equation}
a(e^\theta -1 ) \geq e^{a \theta} -1, \ \forall \theta \in \R.
\label{prop 3.1}
\end{equation}

\begin{theorem}
\label{thm: subcritical}
Let $\alpha = \mu \lambda - 1 - \log (\mu \lambda)> 0$. For $a > 1 / \alpha u$, we have
\[
\prob \left[ \max_{x \in V} |\mathcal{C}_x | \geq a \log n \right] \to 0
\]
as $n \to \infty$.
\end{theorem}

\begin{proof}
Let $\theta \in \R$. Since binomial distributions are the sum of i.i.d. Bernoulli random variables, we have from (\ref{prop 3.1}) and (\ref{eigenvector})
\begin{eqnarray}
\E \left[\exp \left(\theta \bm{u} (\bm{S}_{t+1} - \bm{S}_t)\right) \mid \F_t \right] 
		&=& \sum_{k=1}^{m} 1_{\{v_{t+1} \in V^{(k)}\}} e^{-\theta u_k} \prod_{l=1}^{m}(1-p_{lk}+e^{\theta u_l}p_{lk})^{a_l n}   \nonumber \\
		&\leq& \sum_{k=1}^{m} 1_{\{ v_{t+1} \in V^{(k)} \}} \exp \left(-\theta u_k + \lambda \sum_{l=1}^{m} a_l c_{lk} (e^{\theta u_l}-1)\right) \nonumber \\
		&\leq& \sum_{k=1}^{m} 1_{\{ v_{t+1} \in V^{(k)}\}} \exp \left(-\theta u_k + \lambda  \sum_{l=1}^{m} a_l c_{lk} u_{l} (e^\theta -1) \right) \nonumber \\
		&=&  \sum_{k=1}^{m} 1_{\{ v_{t+1} \in V^{(k)}\}} \exp \left(-\theta u_k + \mu \lambda u_k (e^\theta -1)\right). \label{ineq: moment}
\end{eqnarray}
Now let \[ \varphi _k (\theta) := \exp \bigl(u_k(-\theta + \mu \lambda (e^\theta -1))\bigr) \] and \( C(\theta ) := \max_{1\leq k \leq m}\varphi _k (\theta ) \). Putting \( \theta _0= - \log (\mu \lambda)>0 \) gives \( \varphi _k (\theta _0) = \exp ( -\alpha u_k) \  \forall k\),  hence \[ C(\theta _0) \leq \exp (- \alpha u) <1. \]

Let $C := C (\theta _0)$. Then $M_t := \exp (\theta _0 \bm{u} \bm{S}_t) / C^{t}$ is a $(\F _t)_{t \geq 0}$-supermartingale. Now suppose that we started the exploration process from a vertex $x \in V^{(k)}$. Since we know $T  < \infty \ \prob_x \textrm{-a.s.}$, the optional stopping theorem implies
\[
e^{u_k \theta _0} = \E_x M_0 \geq \E_x M_T = \E_x C^{-T},
\]
hence
\[
\prob_x [T \geq k] \leq C^k \E_x C^{-T} \leq C^k e^{u_k \theta _0} \leq e^{- \alpha u k} e^{\theta_0}.
\]
Now for any $\epsilon >0$,
\begin{equation}
\prob \left[|\mathcal{C} _x| \geq \frac{(1+ \epsilon ) \log n }{ \alpha u }\right] \leq \prob_x \left[T \geq \frac{(1+ \epsilon ) \log n }{ \alpha u} \right] \leq n^{-1-\epsilon} e^{\theta _0}. \label{ineq: component}
\end{equation}
Therefore, for each $a > 1/\alpha u$,  there exists $\epsilon >0$ such that
\[
\prob \left[ \max_{x \in V} |\mathcal{C}_x | \geq a \log n \right] = O(n^{-\epsilon})
\]
and our conclusion follows.
\end{proof}


\subsection{The supercritical case: $\mu \lambda> 1$}

The supercritical case is more complicated than the previous subcritical case. In \cite{Durrett2006}, one of the most crucial tools is a large deviations lemma which guarantees that there are enough active vertices w.h.p. However, this lemma is not immediately applicable to the inhomogeneous model. Therefore, we first start with a large deviations lemma that has limited conclusions but still sufficient to apply to our inhomogeneous model. The following lemma is an observation on the maximal ratio of first taking the exponential and then taking a linear combination versus taking a linear combination first and then applying the exponential. 

\begin{lemma}
\label{lemma: ratio of exponential}
Let $a \neq b$, and consider \[f_{a,b}(p):=\frac{p e^a + (1-p) e^b}{e^{pa+(1-p)b}}.\] Then, we have
\[
\max_{0 \leq p \leq 1} f_{a,b}(p) = f_{a,b}(p_0) = \frac{\frac{e^{b-a}-1}{b-a}}{\exp(b-a+1-\frac{b-a}{1-e^{-(b-a)}})}
\]
where $p_0 = \frac{1}{1-e^{-(b-a)}}-\frac{1}{b-a} \in [0,1]$.
\end{lemma}

We remark that the maximal ratio only depends on the difference $b-a$. Note that the convexity of the exponential function immediately gives $f(p) \geq 1$, but what we need is an upper bound on $f(p)$ to prove the following large deviations lemma.

\begin{lemma}
\label{lemma: large deviation}
Let $\delta >0$ and $Z = X_1 + \dots + X_t$ for
\[
X_i = \sum_{k=1}^{m} 1_{B_{ik}}\bm{u}\bigl(\Bin{a_1 n(1-\delta )}{p_{k1}},\dots, \Bin{a_m n (1-\delta)}{p_{km}}\bigr)^\dagger,
\]
where $\{B_{ik}: k=1, \dots, m\}$ is a partition of the probability space $\Omega$ for each $i = 1, \dots, t$, and all the binomial random variables appearing in the $X_i$'s are independent. Additionally assume that $\sigma(B_{ik}: 1 \leq i \leq t,\ 1 \leq k \leq m)$ and the binomial random variables are independent, and $X_1, \dots, X_t$ are independent as well.

Put $\gamma(x) := x \log x -x +1$ that satisfies $\gamma (x) >0$ for all $x \neq 1$. Let $x<1<y$. Then we have
\[
\prob[Z \leq x \E Z] \leq G(x)^t
\]
and
\[
\prob[Z \geq y \E Z] \leq G (y)^t
\]
for some function $G: \R_{>0} \to \R_{> 0} $ which satisfies $G(x)<1$ for $x<1$ near $1$, and $G(y)<1$ for  $y\gg 1$.
\end{lemma}

\begin{proof}
	For $\theta \in \R$, following (\ref{ineq: moment}) and taking expectations leads to
	\begin{equation}
	\E \exp(\theta X_i) \leq \sum_{k=1}^{m} \prob(B_{ik})\exp \left(\mu \lambda (1-\delta)u_k(e^\theta -1)\right)
	\label{ineq: moment 2}
	\end{equation}
	where we used the independence of $\sigma(B_{ik}: 1 \leq i \leq t,\ 1 \leq k \leq m)$ and the binomial random variables.
	Note that \[\E X_i = \mu \lambda (1-\delta) \sum_{k=1}^{m} \prob(B_{ik})u_k.\] For simplicity we shall write $U_k = \mu \lambda (1-\delta)u_k$. 
	
	(i) Let $x<1$ and $\theta <0$. By Markov's inequality, (\ref{ineq: moment 2}) and the independence of $X_i$'s,
	\begin{eqnarray}
	\prob[Z \leq x\E Z] \leq \frac{\E (\exp \theta Z)}{\exp(\theta x \E Z)} 
	&\leq& \frac{\prod_{i=1}^{t} \Bigl( \sum_{k=1}^{m} \prob(B_{ik})\exp(U_k (e^\theta -1))\Bigr) }{\exp(\theta x \E Z)}\nonumber \\
	&=& \prod_{i=1}^{t} \frac{\sum_{k=1}^{m} \prob(B_{ik})\exp(U_k (e^\theta -1))}{\exp\Bigl(\theta x \sum_{k=1}^{m} \prob(B_{ik})U_k\Bigr)}. \label{lemma: 1st markov}
	\end{eqnarray}
	Putting $\theta = \log x$ in (\ref{lemma: 1st markov}) we have
	\begin{equation}
	\label{lemma: 2nd markov}
	\prob[Z \leq x \E Z] = \prod_{i=1}^{t} \exp\Bigl(-\gamma(x) \sum_{k=1}^{m} \prob(B_{ik})U_k\Bigr) \times \prod_{i=1}^{t} \frac{\sum_{k=1}^{m}\prob(B_{ik})\exp(U_k (x-1))}{\exp \Bigl(\sum_{k=1}^{m} \prob(B_{ik})U_k (x-1)\Bigr)}.
	\end{equation}
	Now let $A = \max \{U_k\}$ and $B = \min \{U_k\}$. Since $\sum_{k=1}^{m} \prob(B_{ik}) =1$ for $1 \leq i \leq t$, the numerator in (\ref{lemma: 2nd markov}) is a linear combination of $\exp \left(A(x-1)\right)$ and $\exp \left( B(x-1)\right)$. Since $\gamma (x) >0$ and $x<1$, from (\ref{lemma: 2nd markov}) we have
	\[
	\prob[Z \leq x \E Z] \leq \exp(-t \gamma (x) B) F(x)^t
	\]
	where \[F(x) = \max_{p \in [0,1]} f_{A(x-1), B(x-1)} (p).\]
	
	Putting 
	\[h(x) = e^{-\gamma (x) B} (e^{(A-B)(1-x)}-1),\] \[g(x) = (A-B)(1-x)\exp\Bigl((A-B)(1-x)+1-\frac{(A-B)(1-x)}{1-e^{-(A-B)(1-x)}} \Bigr)
	\]
	and differentiating gives \[h'(1) = g'(1) = -(A-B),\] \[h''(1) = g''(1) = (A-B)^2,\] \[h'''(1) = -(A-B)^3 + 3B(A-B), \ g'''(1) = -(A-B)^3.\]
	This implies $g(x)-h(x)>0$ for $x<1$ near $1$, and thanks to Lemma~\ref{lemma: ratio of exponential} \[G(x) := \exp(-\gamma (x) B)F(x) = \frac{h(x)}{g(x)}<1\] for $x<1$ near 1, and we have
	\[
	\prob[Z \leq x \E Z] \leq G(x)^t
	\]
	as desired.
	
	(ii) For $y>1$, by a similar argument to (i) with $\theta = \log y$ gives
	\[
	\prob[Z \geq y \E Z] \leq \exp(-t \gamma (y) B) F(y)^t ,
	\]
	and
	\[
	\prob[Z \geq y \E Z] \leq G(y)^t.
	\]
	It is readily verified that $G(y) \to 0$ as $y \to \infty$, and in particular we must have $G(y)<1$ for $y \gg 1$.
\end{proof}

As introduced in \cite{Durrett2006}, the main idea to prove the existence of a giant component is to 
\begin{enumerate}
\item show that a component of order $\geq \beta \log n$ must have order $\Omega_C(n^{2/3})$,
\item then show that if two clusters are sufficiently large they must intersect w.h.p. 
\end{enumerate}
For the homogeneous model, the second part is easily done since the probability that two clusters of size at least $n^{2/3}$ do not intersect is \[(1-p)^{n^{4/3}} \sim \exp(-pn^{4/3}) = \exp(-O(n^{-1/3})) = o(1).\] But in our situation we have the possibility that $p_{kl}=0$ for some $k,l$ and that the two big clusters consist only of type $k$ vertices and type $l$ vertices respectively. In order to solve this issue, we observe that for any $l' \in \{ 1, \dots, m\}$ satisfying $c_{kl'}>0$, the first big cluster consisting of type $k$ vertices must be connected to a large number of type $l'$ vertices. 
The next two lemmas are precise statements of this observation.

\begin{lemma}
\label{lemma: changing vertex type}
Let $V, W$ be sets of fixed labelled vertices of size $\beta n^{2/3}$, $\beta' n$, respectively ($\beta, \beta' >0$). Let $p= \frac{c}{n} (c>0)$ be the probability of a vertex in $V$ and a vertex in $W$ being adjacent. We assume that all the edges $\eta_{x,y}$ are independent for all $x \in V,\  y \in W$. Letting \[Y = \# \{ y \in W: \exists x \in V \textrm{ s.t. } x, y \textrm{ are adjacent}\},\] we have $Y \sim \beta \beta' c n^{2/3}$.
\end{lemma}

\begin{proof}
For each $y \in W$, define $Y_y = 1$ if $y$ is connected to $V$ and $Y_y =0$ otherwise. Then
\[
\prob[Y_y = 0] = (1-p)^{|V|} \sim \exp (-\beta c n^{-1/3}).
\]
Note that the $Y_y$'s are independent and $Y = \sum_{y \in W} Y_y$. Hence
\[
\E Y = \beta' n(1-(1-p)^{|V|}) \sim \beta \beta' c n^{2/3}
\]
and
\[
\textrm{var}(Y) = \beta'n(1-(1-p)^{|V|})(1-p)^{|V|} \sim \beta \beta' c n^{2/3}.
\]
Applying Chebyshev gives
\[
\prob\left[|Y - \E Y | \geq (\log n)(\E Y)^{1/2}\right] = o(1).
\]
\end{proof}

By using the same methods as above, we also have the following.

\begin{lemma}
\label{lemma: changing vertex type 2}
Let $V, W$ be sets of fixed labelled vertices of size $\beta n$, $\beta' n$, respectively ($\beta, \beta' >0$). Let $p= \frac{c}{n} (c>0)$ be the probability of a vertex in $V$ and a vertex in $W$ being adjacent. We assume that all the edges $\eta_{x,y}$ are independent for all $x \in V,\  y \in W$. Letting \[Y = \# \{ y \in W: \exists x \in V \textrm{ s.t. } x, y \textrm{ are adjacent}\},\] we have $Y \sim \beta' n (1-e^{-\beta c})$.
\end{lemma}

We finally prove part (ii) of Theorem~\ref{main thm: component}. 
We follow the steps in \cite{Durrett2006}, but to present a more rigorous and self-contained explanation, we prove the theorem in 5 steps.
	\begin{step}
	Fix a $\delta >0$ such that  $\mu \lambda (1-\delta) >1$. There exists some $\gamma >0$ such that the lower exploration process from Section~\ref{subsec: new exploration} starting with at least $\gamma \log n$ active vertices (i.e. $|\hat{A}_0 | \geq \gamma \log n$) never dies out w.h.p. More precisely, \[\prob \left[\bm{W}_t = \bm{0} \textrm{ for some } t\right] \leq n^{-2}\] for all processes starting with $|\hat{A}_0 | \geq \gamma \log n$.
	\end{step}
	
	\begin{proof}
		We use a similar approach used in the proof of Theorem~\ref{thm: subcritical}. Let $\theta \in \R$ and follow (\ref{ineq: moment}) to obtain
		\[ \E \left[\exp \left( \theta \bm{u} (\bm{W}_{t+1}-\bm{W}_{t})\right) | \F _t \right] \leq  \sum_{k=1}^{m}1_{\hat{v}_{t+1} \in V^{(k)}} \exp \left(-\theta u_k + \mu \lambda (1-\delta)(e^\theta-1)u_k \right). \]
		Now let 
		\[ \varphi _k (\theta ) := \exp \Bigl(u_k(-\theta + \mu \lambda (1-\delta) (e^\theta -1))\Bigr) \] 
		and \( C(\theta ) := \max_{1 \leq k \leq m} \varphi_k (\theta )\). Since 
		\[\varphi_k '(\theta) = (-u_k + u_k \mu \lambda (1-\delta)e^\theta )\varphi_k (\theta),\] 
		we have $\varphi_k ' (0) = u_k (-1 + \mu \lambda (1-\delta)) >0$ for $k=1,\dots, m$. Hence there exists $\theta _{\delta} >0$ such that $C(-\theta_{\delta}) \leq 1$. Then $M_t := \exp(-\theta_{\delta} \bm{u}\bm{W}_t^\dagger)$ is a supermartingale w.r.t. $(\F_t)_{t \geq 0}$. 
		We now define $T_0 := \textrm{inf}\{t: \bm{W}_t = \bm{0} \}$ and fix $x \in V^{(k)}$. By the optional stopping time theorem,
		\begin{equation}
			e^{-\theta_{\delta} u_k} = \E_x M_0 \geq \E_x M_{t \wedge T_0}
			\geq \E_x \left[ \exp(-\theta_{\delta} \bm{u}\bm{W}_{T_0}^\dagger) ; T_0 \leq t \right]
			=\prob_x[T_0 \leq t]	\label{ineq: step 1}
		\end{equation}
		where as usual the subscript $x$ in $\E_x$ and $\prob_x$ means that the lower exploration process starts with $x$. In the followings, $\E_l \ (l \in \N)$ and $\prob_l$ shall denote the expected value and probability for the exploration process starting with $l$ vertices, where any type combination is possible. This is an abuse of notation; to be precise the initial $\hat{A}_0$ is fixed for each probability space and $\E_l$ may have different values according to the type combination of the initial state. However, the inequalities below are true regardless of the type combination.
		
		By following the same argument as in (\ref{ineq: step 1}) and letting $t \to \infty$, we have
		\[
		\prob_l [T_0 < \infty ] \leq e^{-\theta_{\delta}ul}.
		\]
		Now setting $\gamma = 2 / u \theta_{\delta} >0$ gives
		\[
		\prob_{\gamma \log n} [T_0 < \infty] \leq n^{-2}
		\]
		as desired.
	\end{proof}
	
	\begin{step}
	There exists $\beta > 0$ such that $\prob_A (0<|A_{\beta \log n}| < \gamma \log n) = o(1)$ for all $A \subset V$ with $|A| = o(n /\log n)$.
	\end{step}
	
	\begin{proof}
	Let \[S_t := \bm{u}\bm{S}_t^\dagger,\ I_t := \bm{u} \sum_{i=0}^{t-1} \indcbmi.\] Then we may apply Lemma~\ref{lemma: large deviation} with $\delta = 0$ and $B_{ik} = \{ v_{i} \in V^{(k)}\}$ for $Z= S_t - S_0 + I_t $. Note that \[ \E Z = \mu \lambda \sum_{i=1}^{t} \sum_{k=1}^{m} u_k \prob [v_{i} \in V^{(k)}], \ \E I_t = \sum_{i=1}^{t} \sum_{k=1}^{m}u_k \prob[v_{i} \in V^{(k)}],\] hence $\E Z = \mu \lambda \E I_t$. Take an appropriate $\epsilon >0$  such that $G(x) <1$ for $x = \frac{1}{\mu \lambda}(\frac{2\epsilon}{u}+1)<1$, where $G$ is from Lemma~\ref{lemma: large deviation}. 
	
	We want to show that $\prob[S_t - S_0 \leq \epsilon t]$ is small by using Lemma~\ref{lemma: large deviation} on $Z$. Now, in the homogeneous case, we have $I_t = t$ deterministically and $\E Z = \mu \lambda t = \mu \lambda I_t$. This allows us to immediately apply the large deviations lemma to $Z$ since \[ \prob[S_t - S_0 \leq \epsilon t] = \prob[Z \leq I_t + \epsilon t] = \prob[Z \leq (1+\epsilon)t].\] However, in our inhomogeneous case, $I_t$ is a genuine random variable and we obviously do not have $\E Z = \mu \lambda I_t$, which is desirable to apply Lemma~\ref{lemma: large deviation}. Therefore we need to apply an additional deviations bound for $I_t$. For that reason we fix a $\delta'>0$ such that $2\delta'<1$. Then we have
	\begin{eqnarray}
	\lefteqn{\prob[S_t - S_0 \leq \epsilon t]} \nonumber \\
	 &=& \prob\left[S_t-S_0 \leq \epsilon t, \ I_t \geq t^{1/2 + \delta'}+\E I_t\right] + \prob\left[S_t-S_0 \leq \epsilon t,\  I_t <  t^{1/2 + \delta'}+\E I_t\right] \nonumber \\
	&\leq& \prob\left[I_t \geq t^{1/2 + \delta'}+\E I_t\right] + \prob\left[Z \leq \epsilon t + t^{1/2 + \delta'} + \E I_t\right]	\nonumber \\
	&\leq& \exp(-2t^{2\delta'} / U^2) + \prob\left[Z \leq \epsilon t + t^{1/2 + \delta'} + \E Z / \mu \lambda\right] \label{ineq: 1st dev}
	\end{eqnarray}
	where the last inequality follows from Hoeffding's inequality. By observing that $ \E Z \geq \mu \lambda u t$ and applying Lemma~\ref{lemma: large deviation}, the last term in (\ref{ineq: 1st dev}) can be estimated by
	\begin{eqnarray}
	\prob\left[Z \leq \epsilon t + t^{1/2 + \delta'} + \E Z / \mu \lambda \right]
	&=& \prob \left[Z \leq \E Z(\frac{\epsilon t + t^{2\delta'}}{\E Z} + \frac{1}{\mu \lambda})\right]	\nonumber \\
	&\leq& \prob \left[Z \leq \E Z\left(\frac{\epsilon t + t^{2\delta'}}{\mu \lambda u t} + \frac{1}{\mu \lambda}\right)\right] \nonumber \\
	&\leq& \prob \left[Z \leq \E Z\left(\frac{2\epsilon t}{\mu \lambda u t} + \frac{1}{\mu \lambda}\right)\right] \nonumber \\
	&=& \prob\left[Z \leq  \frac{1}{\mu \lambda }\left(\frac{2\epsilon}{u}+1\right) \E Z\right] \to 0 	\label{ineq: 2nd dev}
	\end{eqnarray}
	as $t \to \infty$. Combining (\ref{ineq: 1st dev}) and (\ref{ineq: 2nd dev}), we obtain
	\begin{equation}
	\label{ineq: 3rd dev}
	\prob[S_t - S_0 \leq \epsilon t] \to 0 \textrm{ as } t \to \infty.
	\end{equation}
	
	Fix a $y>1$ such that $G(y)<1$. Again by Lemma~\ref{lemma: large deviation},
	\begin{equation}
	\label{ineq: 4th dev}
	\prob\left[S_t -S_0 + I_t \geq y \mu \lambda U t\right] \leq \prob\left[Z \geq y \E Z\right] \to 0 \textrm{ as } t \to \infty. 
	\end{equation}
	
	We note that $S_t - S_0+I_t<y\mu \lambda Ut$ corresponds to 
	\[\bm{u}\bm{U}_t^\dagger > n\bm{u}\bm{a}^\dagger - y \mu \lambda U t - S_0\] 
	as $\bm{U}_t + \bm{W}_t + \bm{I}_t = n \bm{a}$ for all $t$. Since  $\bm{U}_s \geq \bm{U}_{s+1}$ for all $s$,
	\begin{eqnarray}
	\lefteqn{\prob_A \left[S_t-S_0+I_t < y \mu \lambda U t,\  \bm{S}_t \neq \bm{A}_t \right]} && \nonumber \\
	 &\leq& u^{-1}\E_A \left[\sum_{s=1}^{t} \sum_{k=1}^{m} 1_{v_s \in V^{(k)}} \bm{u} \left(\Bin{a_l n - |U_s ^{(l)}|}{p_{lk}} \right)_{1 \leq l \leq m}^\dagger ; S_t-S_0+I_t < y \mu \lambda U t \right] \nonumber \\
	 &\leq& u^{-1} \sum_{s=1}^{t} \sum_{k=1}^{m} \sum_{l=1}^{m} u_l (a_l n - |U_s^{(l})|)p_{lk} \textrm{ with } \bm{u}\bm{U}_t^\dagger > n\bm{u}\bm{a}^\dagger - y \mu \lambda U t - S_0 \nonumber \\
	 &\leq&  C u^{-1} \sum_{s=1}^{t} \sum_{k=1}^{m} (y \mu \lambda U t + S_0)  \nonumber \\
	 &\leq& \frac{C  m t }{nu}(y \mu \lambda U t + |A|) \label{ineq: 5th dev}
	\end{eqnarray}
	for $C= \max \{c_{kl} : 1 \leq k, l \leq m \}$.
	Now choose $\beta > 0$ such that $\epsilon \beta > \gamma$ where $\gamma$ is from \textbf{Step 1}.
	Put $r = \beta \log n$. Combining (\ref{ineq: 3rd dev}), (\ref{ineq: 4th dev}), and (\ref{ineq: 5th dev}), we have
	\begin{eqnarray}
	\prob_A [0<|A_r | \leq \epsilon r] &\leq& \prob_A [S_r-S_0+I_r \geq y \mu \lambda U r]+ \prob_A [S_r -S_0+I_r < y \mu \lambda U r, \ \bm{S}_r \neq \bm{A}_r] \nonumber \\
	&& \mbox{ } + \prob_A [0<|A_r | \leq \epsilon r, \ \bm{S}_r = \bm{A}_r] \nonumber \\
	&\leq& o(1) + \prob_A [S_r -S_0 \leq \epsilon r]=o(1). \nonumber 
	\end{eqnarray}
	\end{proof}
		
	\begin{step}
	All clusters reaching size $\beta \log n$ intersect w.h.p.
	\end{step}
	
	\begin{proof}
	We define $W_t := \bm{u} \bm{W}_t^\dagger$. 
	By similar calculations to (\ref{ineq: 3rd dev}), (\ref{ineq: 4th dev}), we have for some $\epsilon_{\delta}>0, y>1$,
	\begin{equation}
	\label{ineq: 1st dev lb}
	\prob[W_t - W_0 \leq \epsilon_{\delta} t ] \to 0,
	\end{equation}
	\begin{equation}
	\label{ineq: 2nd dev lb}
	\prob[W_t - W_0 + I_t \geq y \mu \lambda (1-\delta) U t] \to 0
	\end{equation}
	as $t \to \infty$.
	Recall from \textbf{Step 2} that we put $r = \beta \log n$. Letting $a = a_1 \wedge \dots \wedge a_m$, (\ref{ineq: 2nd dev lb}) implies
	\[
	\prob\left[W_{n^{2/3}}-W_0+I_{n^{2/3}} \geq \delta au n - r \right]\to 0
	\]
	as $n \to \infty$, and since $W_s -W_0 +I_s$ is non-decreasing with respect to $s \geq 0$, we have w.h.p.
	\begin{equation}
	\label{ineq: coupling}
	W_s - W_0 + I_s \leq \delta au n - r,\ \forall s \leq n^{2/3}.
	\end{equation}

	For $x \in V$, suppose that $\C_x$ has size $\geq r.$ Executing the upper exploration process with  $A_0 = \{ x \}$, we have $|A_r | > 0$. Then \textbf{Step 2} implies that $|A_r | > \epsilon r$ w.h.p. After these $r$ steps of exploration we forget how the vertices were marked, i.e. whether they were active, removed, or unexplored, but instead carry out the lower exploration process starting from $\hat{A}_0 \subset A_r$ with $|\hat{A}_0|=r$. Note that (\ref{ineq: coupling}) along with $W_0 \leq r$ implies that we have at least $|V| - \delta a n$ unexplored vertices. Hence by \textbf{Step 1} and (\ref{ineq: coupling}), we have $T_W > n^{2/3}$ and $\bm{W}_s = \hat{\bm{A}}_s \ \forall s \leq n^{2/3}$ w.h.p.
	
	In other words, the coupling between $\bm{W}_s$ and $\hat{\bm{A}}_s$ remains valid for $s \leq n^{2/3}$ w.h.p.. Combining with (\ref{ineq: 1st dev lb}), if a cluster reaches size $\geq r$, then the set of active sites in the lower exploration process has size $> \epsilon_{\delta} n^{2/3}$ w.h.p.
	
	Now suppose that there are two components $\C_1, \C_2$ of size $\geq r$.  Then w.h.p. each has at least size $\epsilon_{\delta} n^{2/3}$, and there exists $k, l \in \{1, \dots, m\}$ such that $\C_1$ has at least $(\epsilon_{\delta}/m) n^{2/3}$ vertices of type $k$ and $\C_2$ has at least $(\epsilon_{\delta}/m) n^{2/3}$ vertices of type $l$. Recalling Remark~\ref{main thm: assumption} and applying Lemma~\ref{lemma: changing vertex type} several times we have w.h.p. at least $d n^{2/3}$ vertices of type $l'$ in $\C_1$, with $c_{ll'}>0$ for some constant $d>0$.
	Hence probability of two clusters of size $\geq r$ being disjoint is $\leq (1-\frac{\lambda c_{ll'}}{n})^{d' n^{4/3}} \leq \exp(-O(n^{-1/3})) = o(1). $
	\end{proof}
	
	\begin{step}
	$\prob[|\mathcal{C}_x| \leq \beta \log n] \to q_k$ for fixed $x \in V^{(k)}$, where $q_k$ is the extinction probability from Definition~\ref{def: associated bp}.
	\end{step}
	
	\begin{proof}
	Let \[T_0 ^{(n)} := \textrm{inf}\{ t: \bm{S}^{(n)} _t= \bm{0} \},\ r_n = \beta \log n.\] We add a superscript $n$ merely to avoid confusion. By the natural correspondence between branching processes and random walks (see e.g. \cite{Roch2024}), $T_0 ^{(n)}$ is the total progeny of a branching process $\bm{Y}^{(n)} _t$ with $m$ types of individuals, each type $k$ having $\Bin{a_1 n}{p_{k1}} \otimes \dots \otimes \Bin{a_m n}{p_{km}}$ as offspring distribution. 
	
	Note that there exists a probability space such that 
	\begin{equation}
	\prob\left[\Bin{l}{c/l} \neq \textrm{Poi}(c) \right] \leq 2 c^2 / l,	\label{binomial and poisson}
	\end{equation}
	a result from general probability theory. We write $\bm{Z} = \bm{Z}^{\G}$ for the branching process defined in Definition~\ref{def: associated bp}. Assume that $\bm{Z}_0 = \bm{Y}^{(n)}_0$ and all the Poisson distributions occuring in $\bm{Z} _t$ satisfy (\ref{binomial and poisson}) with respect to the binomial distributions occuring in $\bm{Y}^{(n)}_t$.
	Then, when conditioned on $T_0 ^{(n)} \leq r_n$,
	\[
	\prob\left[\bm{Y}_{t+1} ^{(n)} = \bm{Z} _{t+1}\right] \geq \prob\left[\bm{Y}_{t+1} ^{(n)}= \bm{Z}_{t+1} | \bm{Y}_t ^{(n)} = \bm{Z}_t\right] \prob\left[\bm{Y}_t ^{(n)}= \bm{Z}_t\right]
	\geq \left(1- C \frac{\log n}{n} \right) \prob\left[\bm{Y}_t ^{(n)}= \bm{Z}_t\right],
	\]
	and we have
	\begin{equation}
	\prob\left[\bm{Y}_{t+1} ^{(n)}= \bm{Z}_{t+1}\right] \geq \left(1- C \frac{\log n}{n} \right)^{t+1} \geq 1 - C(t+1)\frac{\log n}{n}
	\label{ineq: poi and bin}
	\end{equation}
	for some constant $C>0$. Let $T_0$ be the total progeny of $\bm{Z}_t$. Thanks to (\ref{ineq: poi and bin}),
	\begin{eqnarray}
	\prob\left[T_0 ^{(n)} \leq r_n \right] &=& \prob\left[T_0 ^{(n)} \leq r_n, \ \bm{Y}_t ^{(n)}= \bm{Z}_t \ \forall t \leq r_n \right] \nonumber \\
	&& \mbox{ } + \prob\left[T_0 ^{(n)} \leq r_n, \ \bm{Y}_t ^{(n)} \neq \bm{Z}_t \ \exists t \leq r_n \right] \nonumber \\
	&\leq& \prob\left[T_0 \leq r_n\right] + \prob\left[ \bm{Y}_t ^{(n)}\neq \bm{Z}_t \ \exists t \leq r_n\  \big| \ T_0 ^{(n)} \leq r_n\right] \nonumber \\
	&\leq& \prob\left[T_0 \leq r_n\right] + \sum_{t=1}^{r_n} \prob\left[\bm{Y}_t ^{(n)} \neq \bm{Z}_t \ \big| \ T_0 ^{(n)} \leq r_n\right] \nonumber \\
	&\leq& \prob\left[T_0 \leq r_n\right] + C \frac{\log n}{n} \sum_{t=1}^{r_n} t
	= \prob\left[T_0 \leq r_n\right] + o(1). \label{ineq: progeny 1}
	\end{eqnarray}
	Similarly, we have 
	\begin{equation}
	\label{ineq: progeny 2}
	\prob\left[T_0 \leq r_n\right] \leq \prob\left[T_0 ^{(n)} \leq r_n\right]+ o(1).
	\end{equation}
	Clearly $\prob\left[T_0 \leq r_n\right] \to \prob\left[T_0 < \infty\right]$, and $\prob[T_0 < \infty]$ is exactly the extinction probability of $\bm{Z}_t$, depending on the initial state.
	
	Now fix a vertex $x \in V^{(k)}$. Since $\bm{S}_t \geq \bm{A}_t$,
	\begin{equation}
	\label{ineq: extinct prob 1}
	\prob\left[|\mathcal{C}_x| \leq r_n\right] \geq \prob_x \left[T_0 ^{(n)} \leq r_n\right] \to q_k
	\end{equation}
	by (\ref{ineq: progeny 1}) and (\ref{ineq: progeny 2}). We proceed by
	\begin{eqnarray}
	\prob\left[|\mathcal{C}_x| \leq r_n\right]&\leq& \prob\left[T_0 ^{(n)} \leq r_n\right]+ \prob\left[T_0 ^{(n)} > r_n, \  |\mathcal{C}_x| \leq r_n\right] \nonumber \\
	&\leq& \prob\left[T_0 ^{(n)} \leq r_n\right] + \prob\left[T_0 ^{(n)} > r_n,\  \bm{S}_{r_n} \neq \bm{A}_{r_n}\right] \nonumber \\
	&\leq& \prob\left[T_0 ^{(n)} \leq r_n\right] + \prob\left[S_{r_n} -S_0 + I_{r_n} < y \mu \lambda U r_n, \bm{S}_{r_n} \neq \bm{A}_{r_n}\right] \nonumber \\
	&& \mbox{ } + \prob\left[S_{r_n} -S_0 + I_{r_n} \geq y \mu \lambda U r_n\right] \nonumber \\
	&=& \prob\left[T_0 ^{(n)} \leq r_n\right] + o(1),
	\label{ineq: extinct prob 2}
	\end{eqnarray}	
	where the last equality owes to (\ref{ineq: 4th dev}) and (\ref{ineq: 5th dev}). Finally (\ref{ineq: extinct prob 1}) and (\ref{ineq: extinct prob 2}) give \[\prob[|\mathcal{C}_x| \leq r_n] \to q_k.\]
	
	\end{proof}
	
	\begin{step}
	The giant component has size $\sim n\sum_{k=1}^{m}a_k (1-q_k)$.
	\end{step}
	
	\begin{proof}
	Define random variables $Y_x := 1_{|\mathcal{C}_x | \leq \beta \log n}$ for each $x \in V$.
	We shall find an upper bound for the covariance of two random variables $Y_x$ and $Y_y\ (x \neq y)$ and use this upper bound to deduce that the giant component size converges in probability. 
	
	In order to estimate the covariance of $Y_x$ and $Y_y$, we consider two exploration processes starting from $x$ and $y$ simultaneously. We first run the upper exploration process with $A_0 = \{ x \}$ until time $t = \beta \log n$. Introduce another upper exploration process with $A' _0 = \{y \}$ that continues according to (\ref{eq: exploration}) until $v' _{t+1} \in R_{\beta \log n}$ for some $t$. If $v' _{t+1} \in R_{\beta \log n}$, from time $t+1$ we follow exactly the exploration process starting from vertex $x$. With this coupled exploration process we deduce that
	\begin{eqnarray}
	\prob[Y_x=1, Y_y=1 ] &=& \prob[\mathcal{C}_x \cap \mathcal{C}_y = \emptyset, Y_x =1, Y_y=1]+ \prob[\mathcal{C}_x \cap \mathcal{C}_y \neq \emptyset, Y_x=1, Y_y=1] \nonumber \\
	&\leq& \prob[Y_x=1]\prob[Y_y=1] + \prob[R_{\beta \log n} \cap R' _{\beta \log n} \neq \emptyset]\nonumber \\
	&\leq& \prob[Y_x=1]\prob[Y_y=1] + \lambda \max_{1 \leq k, l \leq m} c_{kl} \frac{(\beta \log n)^2}{n}. \label{ineq: collision prob}
	\end{eqnarray}
	Now we have
	\[
	\textrm{var} \left( \sum_{x \in V} Y_x \right) \leq N +\lambda C {N \choose 2}\frac{(\beta \log n)^2}{n} = O(n(\log n)^2)
	\]
	where $N = |V| = \sum_{k=1}^m a_k n$. Chebyshev's inequality gives
	\begin{equation}
	\prob\left[\left| \sum_{x \in V} (Y_x - \E Y_x ) \right| \geq n^{2/3} \right] = \frac{O(n(\log n)^2)}{n^{4/3}} \to 0. \label{ineq: Z}
	\end{equation}
	Put $Z := \sum_{x \in V} Y_x$. Then $N-Z$ is the size of the giant component w.h.p. by \textbf{Step 3}. (\ref{ineq: Z}) implies
	\[
	\prob\left[\left| \frac{1}{n} Z - \frac{1}{n}\sum_{x \in V} \E Y_x \right| \geq n^{-1/3}\right] \to 0,
	\]
	and we have \[ \frac{1}{n}\sum_{x \in V} \E Y_x  \to \sum_{k=1}^{m} a_k q_k \] by \textbf{Step 4}. Thus the giant component, which exists w.h.p., has size $\sim n \sum_{k=1}^{m}a_k (1-q_k)$, and the second largest component has size $\leq \beta \log n$ w.h.p.
	
	Recall that we only had $q_k<1 \ \forall k$ for positively regular $\Lambda$ (Corollary~\ref{cor: mbp}). But we immediately have $q_k<1 \ \forall k$ for $\Lambda$ only satisfying our base assumption Remark~\ref{main thm: assumption} by Lemma~\ref{lemma: changing vertex type 2}. 
	\end{proof}
	
We would like to emphasize that in the supercritical regime, every vertex type has a meaningful contribution to the giant component, i.e. there are $\sim a_k(1-q_k)n$ vertices of type $k$ in the giant component for all $k= 1, \dots, m$.


\section{Connectivity}

The homogeneous case (Theorem~\ref{thm: homo connected}) and the general case (Theorem~\ref{alon thm}) hints that we would  expect the threshold for connectivity to be $p(n) = \Theta(\frac{\log n}{n})$. Before we prove our main theorem on connectivity, Theorem~\ref{main thm: connectivity 1}, we introduce new notations for conciseness. We additionally define $\lambda = \lambda(n) = np(n)$ and use the matrix
\[
\Lambda '= (a_l n p_{kl} )_{1 \leq k, l \leq m} = \lambda \Lambda.
\]
Then $\Lambda$ is a matrix of constant entries, and $\Lambda '$ is a matrix representing the expected number of offsprings of a vertex of a particular type. Let $\mu, \mu '$ be the Perron roots of $\Lambda, \Lambda'$ respectively. Then we have $\mu'= \mu \lambda$, and there exists a vector $\bm{u}$ that is an row eigenvector for $(\Lambda, \mu)$ and $(\Lambda', \mu')$.

We shall assume Remark~\ref{main thm: assumption} as in the previous section. Then we have Proposition~\ref{prop: u>0}, hence \[U = \max \{ u_1 , \dots , u_m \} >0 \textrm{ and }u = \min \{ u_1 , \dots , u_m \} >0.\] We also assume $U \leq 1$. 

We give a brief motivation on why the critical value is $p = \log n / bn$ by noticing the similarities between a homogeneous graph model and a inhomogeneous model. We first turn to the homogeneous model and observe the behavior of $\G_{n,p}$ for $p = \log n /n$. Theorem~7.3 of \cite{Bollobas2001} calculates for \[p = \frac{\log n + c + o(1)}{n} \ (c \in \R)\] that $\G_{n,p}$ consists of a giant component and isolated vertices. This is done by calculating the expected number of components of size $2 \leq k \leq n/2$. The same calculations are applicable in our inhomogeneous situation with \[p = \frac{\log n  + c + o(1)}{bn} \ ( c \in \R )\] and $G$ consists of a giant component and isolated vertices. This implies that the existence of isolated vertices is the essential cause for not being connected. On the other hand, it is natural to associate the expected value of the degree of a fixed vertex with the probability that it is isolated. This accounts for the appearance of $b$ in the critical value.


\subsection{The subcritical case}
We first give a proof of (i) in Theorem~\ref{main thm: connectivity 1}. We deal with this subcritical case by counting the number of isolated vertices as in \cite{Durrett2006}. Note that every vertex is equivalent in the homogeneous case, but for the inhomogeneous case, it is not guaranteed that every $V^{(k)}$ contains isolated vertices. 

\begin{proof}
Assume that $\alpha b < 1$. Then there exists $l \in \{1, \dots, m\}$ such that $\alpha b_l <1$. Fix a vertex $x \in V^{(l)}$ and let $d_x$ be the degree of $x$. We have 
\begin{equation}
\prob[d_x = 0] = \prod_{k=1}^{m} \left(1-\frac{\alpha c_{lk} \log n}{n}\right)^{a_k n}
\sim \exp \left(-\alpha \log n \sum_{k=1}^{m} a_k c_{kl} \right) 
= n^{-\alpha b_l}.	\label{eq: degree 0}
\end{equation}

Now let $I_n$ denote the number of isolated vertices of type $l$. By (\ref{eq: degree 0}),
\begin{equation}
\label{eq: EI_n}
\E I_n \sim a_l n^{1-\alpha b_l} \to \infty.
\end{equation}
We calculate the variance of $I_n$ by first observing that for $x, y \in V^{(l)}, x \neq y$,
\[
\prob[d_x = 0, d_y = 0]= \left(1- \frac{\alpha c_{ll}\log n}{n}\right)^{-1} \prob[d_x = 0]\prob[d_y = 0],
\]
then
\begin{eqnarray}
\textrm{var}(I_n) &=& a_l n \prob[d_x = 0 ]\left(1-\prob[d_x= 0]\right) 	\nonumber \\
	&& +  a_l n (a_l n -1)\left(\left(1-\frac{\alpha c_{ll}\log n}{n}\right)^{-1}-1\right) \prob[d_x = 0]\prob[d_x = 0] \nonumber \\
&\sim& a_l n n^{-\alpha b_l} + (a_l n )^2 \frac{\alpha c_{ll}\log n}{n} n^{-2\alpha b_l}	\nonumber \\
&\sim& a_l n^{1-\alpha b_l} + C n^{1 - 2\alpha b_l}\log n	\nonumber \\
&\sim& a_l n^{1-\alpha b_l}.	\label{eq: var(I_n)}
\end{eqnarray}
By (\ref{eq: EI_n}) and (\ref{eq: var(I_n)}), 
\begin{equation}
\prob \left[ | I_n - \E I_n | > (\log n) (\E I_n)^{1/2} \right] \leq \frac{\textrm{var}(I_n)}{(\log n)^2 (\E I_n)} \sim 1/(\log n)^2 =o(1).
\end{equation}
This implies that $G$ has $\sim a_l n^{1-\alpha b_l} \to \infty$ isolated vertices of type $l$, and in particular $G$ cannot be connected w.h.p.
\end{proof}


\subsection{The supercritical case}
For the supercritical case, we use similar ideas used in the emergence of a giant component. We shall again execute exploration processes and use a large deviations bound.

We had \(n p_{kl} = \lambda c_{kl}\) constant in our previous section and this allowed the function $G(x)$ in Lemma~\ref{lemma: large deviation} to be independent of $n$, i.e. for any $n$, $\prob [Z \leq x \E Z ] \leq G(x)^t$ held for a fixed function $G(x)$. But in our situation, $n p_{kl} = c_{kl} \lambda (n)$ leads to an upper bound $\prob[Z \leq x \E Z] \leq G(x, n)^t$ where $G(x,n)$ is a function that depends on $n$ and does not work very well. Hence we need another lemma that will play a similar role with Lemma~\ref{lemma: large deviation}. The proof is inspired from the one in \cite[Lemma 2.3.3]{Durrett2006}.

\begin{lemma}
\label{lemma: 2nd easy dev}
Under the same settings of Lemma~\ref{lemma: large deviation}, we have for $x<1<y$,
\[
\prob[Z \leq x \mu \lambda (1-\delta) u t] \leq \exp\left(-\gamma(x) \mu \lambda(1-\delta)ut\right)
\]
and
\[
\prob[Z \geq y \mu\lambda (1-\delta)Ut] \leq \exp\left(-\gamma(y)\mu \lambda(1-\delta)Ut\right).
\]
\end{lemma}

\begin{proof}
We first prove the lemma for $x<1$. Let $\theta <0$. Then (\ref{ineq: moment 2}) gives \[ \E \exp \theta X_i \leq \exp \left(\mu \lambda (1-\delta) u (e^\theta -1)\right), \] hence \[ \E \exp \theta Z \leq \exp\left(\mu \lambda (1-\delta)u(e^\theta -1)t\right).\] Applying Markov's inequality with $\theta = \log x$ gives
\[
\prob\left[Z \leq x \mu \lambda (1-\delta)ut\right] \leq \exp \left(-\gamma(x) \mu \lambda (1-\delta)ut\right).
\]

Similarly for $y>1$, letting $\theta >0$ we have \[ \E \exp \theta X_t \leq \exp\left(\mu \lambda (1-\delta) U (e^\theta -1)\right), \] and applying Markov with $\theta = \log y$ gives
\[
\prob[Z \geq y \mu \lambda (1-\delta)Ut] \leq \exp \left(\gamma(y) \mu \lambda (1-\delta)Ut\right).
\]
\end{proof}

For the homogeneous case, \cite{Durrett2006} observes that a vertex $x \in V$ must have at least 14 neighbors w.h.p. From these neighbors it is guaranteed that at least $7 \log n$ vertices are within distance 2 from $x$ w.h.p., and by using \textbf{Step 1} in the proof of Theorem~\ref{main thm: component} the lower exploration process is valid for a long time and $\C_x$ has order $\geq 0.1 n^{1/2}\log n$ w.h.p. We follow similar steps, and subtle differences arise as in the proof of the giant component. 

We begin the proof of Theorem~\ref{main thm: connectivity 2} (ii). Let $\alpha b >1$. 

\begin{proof}
Let $x \in V^{(k)}$. For a fixed integer $K$, we want to find an upper bound for $\prob[d_x \leq K]$, and thus calculate
\begin{eqnarray}
\prob[d_x \leq K] &=& \sum_{0 \leq i_1 + \cdots + i_m \leq K} \left[ \ \prod_{l=1}^{m} {a_l n \choose i_l} \left(\frac{\alpha c_{kl}\log n}{n}\right)^{i_l} \left(1-\frac{\alpha c_{kl}\log n}{n}\right)^{a_l n - i_l}	\right] \nonumber \\
&\leq& \sum_{0 \leq i_1 + \cdots + i_m \leq K} \left[\ \prod_{l=1}^{m} (a_l \alpha c_{kl})^K (\log n)^{i_l} \exp \left( -\alpha a_l c_{kl} \log n + \frac{\alpha c_{kl} K \log n}{n}\right) \right] \nonumber \\
&= & O\left((\log n)^K n^{-\alpha b_k}\right) = o(n^{-1}).	\label{ineq: degree>K}
\end{eqnarray}

Now take an integer $K$ such that $\mu \alpha u K > 2$ and $\gamma(1/2)\mu \alpha u K >2$. Let $Y_x$ be the number of vertices at distance $\leq 2$ from $x$. If $d_x >K$, let $v_1, \dots , v_K$ denote different vertices that are adjacent to $x$, and define $Z = X_1 + \cdots + X_K$ where $X_i$ is the number of vertices adjacent to $v_i$. We say that there is a `collision' in $\{y: \textrm{dist}(y,x) \leq 2\}$ if there exists a vertex $\neq x$ that is adjacent to both $v_i, v_j$ for some $1 \leq i < j \leq K$. Then applying Lemma~\ref{lemma: 2nd easy dev} with $\delta = 0, x = 1/2$ gives
\begin{eqnarray}
\prob[Y_x \leq \log n]&\leq& \prob[d_x \leq K] + \prob[d_x > K, Y_x \leq \log n] \nonumber \\
&\leq& o(n^{-1}) + \prob[Z \leq \log n + K^2 ] \nonumber \\
	&&+ \prob\left[\textrm{exists a collision in }\{y: \textrm{dist}(y,x) \leq 2\}, Y_x \leq \log n \right]\nonumber \\
&\leq& o(n^{-1}) + n^{-\gamma(1/2)\mu \alpha uK} + o(n^{-1}) 
= o(n^{-1})	.	\label{ineq: Y_x}
\end{eqnarray}

Fix $\delta \in (0,1)$, e.g. $\delta = 1/2$. Consider the random walk $\bm{W}_t$ from the lower exploration process and let $T_0 = \inf \{t:  \bm{W}_t = \bm{0} \}$. Following \textbf{Step 1} in the proof of Theorem~\ref{main thm: component}, we see that for any $\theta >0$, 
\[
 \varphi_k (-\theta) := \exp \Bigl(u_k \bigl(\theta + \mu' (1-\delta)(e^{-\theta} -1)\bigr) \Bigr) <1
\] 
for sufficiently large $n$ since $\mu' = \mu \lambda = \mu \alpha \log n$. In particular, letting $\theta = 2/u$ gives
\begin{equation}
\label{ineq: T_0}
\prob_{\log n}[T_0 < \infty] \leq n^{-2} = o(n^{-1}).
\end{equation}
Now run the lower exploration process starting from \(\{y \in V : \textrm{dist}(y,x) \leq 2\}\). Combining (\ref{ineq: Y_x}) and (\ref{ineq: T_0}) gives $T_0 = \infty$ with probability $1 - o(n^{-1})$. Take a constant $\beta < \mu \alpha u /5$. Following (\ref{ineq: 1st dev}) from \textbf{Step 2} of Theorem~\ref{main thm: component} for $t = n^{1/2}$ and applying Lemma~\ref{lemma: 2nd easy dev} with $\delta= x = 1/2$ gives
\begin{eqnarray}
\prob[W_t - W_0 \leq 2\beta t \log t ] &\leq& \exp(-2n^{\delta '} / U^2) + \prob\left[Z \leq 2\beta t \log t + t^{1/2 + \delta '} +\E Z / \mu' \right] \nonumber \\
&\leq& o(n^{-1}) + \prob\left[Z \leq 2\beta t\log t + t^{1/2 + \delta '} + Ut \right] \nonumber \\
&\leq& o(n^{-1}) + \prob\left[Z \leq \frac{\mu' u t}{2} \left(\frac{4 \beta \log t}{\mu' u}+\frac{2 t^{1/2 + \delta'}}{\mu' ut} + \frac{2U}{\mu' u}\right) \right] \nonumber \\
&\leq& o(n^{-1}) + \prob \left[Z \leq \frac{\mu' ut}{2}\left(\frac{2\beta}{\alpha \mu u}+o(1)\right)\right] \nonumber \\
&\leq& o(n^{-1}) + \prob \left[Z \leq \frac{\mu' ut}{4}\right] \nonumber \\
&\leq& o(n^{-1}) + \exp \left( -\gamma(1/2)\mu ut/2 \right)  = o(n^{-1}). 	\label{ineq: many active}
\end{eqnarray}
By putting $\delta = 0,\ y=2$ in Lemma~\ref{lemma: 2nd easy dev},
\begin{equation}
\prob[W_t - W_0 + I_t  \geq 2\alpha \mu U t \log t] \leq \exp \left(-\gamma(2) \mu U t/2\right) = o(n^{-1}),
\label{ineq: valid coupling}
\end{equation}
and along with the fact that $T_0 = \infty$ with probability $1-o(n^{-1})$, the coupling between $\bm{W}_s$ and $\hat{\bm{A}}_s$ remains valid for $s \leq n^{1/2}$ with probability $1-o(n^{-1})$. By (\ref{ineq: many active}), at least $\beta n^{1/2}\log n$ vertices are active at time $n^{1/2}$ with probability $1-o(n^{-1})$. 

Denote the largest component by $\C_1$. Suppose that at time $n^{1/2}$ there are at least $\beta n ^{1/2}\log n$ active vertices in our lower exploration. Then there exists $l$ such that there are at least $\frac{\beta}{m}n^{1/2}\log n$ active vertices of type $l$. By our assumption in Remark~\ref{main thm: assumption} there exists $l'$ such that $c_{ll'}>0$, and Theorem~\ref{main thm: component} gives that for some $\beta '>0$, the largest component has at least $\beta' n^{1/2}\log n $ vertices of type $l'$. Then the probability that $\C_1$ and $\C_x$ are disjoint is
\[
\leq \left(1-\frac{\alpha c_{ll'}\log n}{n}\right)^{d n(\log n)^2} \sim \exp\left(-d c_{ll'}(\log n)^3\right) = o(n^{-1}).
\]
Hence
\[
\prob[G \textrm{ is not connected}]=\prob[ \exists x \in V \textrm{ s.t. } \C_1\cap \C_x = \emptyset] = N o(n^{-1}) = o(1)
\]
where $N = |V|$.
\end{proof}


\subsection{The critical case}
Let $\lambda = (\log n + c + o(1))/b$ throughout this subsection. We prove the first part of Theorem~\ref{main thm: critical connectivity} by observing that there are no components of size between 2 and $\beta n$ for some $\beta >0$. Let $C = \max_{1 \leq k, l \leq m }c_{kl}$.
\begin{lemma}
\label{lemma: connectivity critical}
Let $q$ be the probability that for some $2 \leq r \leq \frac{b}{3C}n$, there exists a component of size $r$. Then $q = o(1)$.
\end{lemma}

\begin{proof}
We may assume $\lambda = (\log n + c)/b$.  Let $q_r$ be the probability that there exists a component of size $r$. Consider the complete graph $K$ formed from the vertex set $V$. We write $T \in K$ to indicate that $T$ is a tree of order $r$ contained in $K$ and $E_T$ shall denote the edges of $K$. We also put $\bm{i}^T = (i_k^T)_{1\leq k \leq m}$ where $i_k^T$ is the number of type $k$ vertices in $T$. We have 
\[ q_r = \sum_{T \in K} \left[ \prod_{e \in E_T} p_e  \prod_{1 \leq k, l \leq m} (1-p_{kl})^{(a_l n - i_l^T)i_k^T}\right], \]
where $p_e$ denotes the probability of the edge $e$ in our model $\G$. Now fix $\bm{i} = (i_k)_{1\leq k \leq m}$ such that 
\[  \prod_{1 \leq k, l \leq m} (1-p_{kl})^{(a_l n - i_l)i_k} = \max_{T \in K} \left[ \prod_{1 \leq k, l \leq m} (1-p_{kl})^{(a_l n - i_l^T)i_k^T} \right]. \]
Note that $|E_T| = r-1$, $|\{ T \in K \}| = r^{r-2}$ and ${n \choose k} \leq \left(\frac{ne}{k}\right)^k$ for every $n, k$. Let $A = \sum_{k=1}^m a_k$ and $N' = \frac{b}{3C}n$. We proceed by
\begin{eqnarray*}
\sum_{r=2}^{N'} q_r &\leq& \sum_{r=3}^{N'} \left[ {N \choose r} r^{r-2} \left( \frac{C \lambda}{n} \right)^{r-1} \prod_{1 \leq k, l \leq m} (1-p_{kl})^{(a_l n - i_l)i_k}\right] \\
&\leq& \sum_{r=2}^{N'} \left[(Ane)^r r^{-2} \left( C \frac{\log n + c}{bn} \right)^{r-1} \exp\left( -\sum_{k, l} p_{kl}(a_l n - i_l)i_k \right)\right] \\
&\leq& \sum_{r=2}^{N'} \exp \Biggl( r( \log n + C') + (r-1)(-\log n + \log \log n + C') \\
&& \mbox{ \ \ \ \ \ \ \ \ \ \ \ \ \ \ } - \frac{\log n + d}{b} \sum_k b_k i_k +  C \frac{\log n + c}{bn} \sum_{k, l} i_k i_l \Biggr) \\
&\leq& \sum_{r=2}^{N'} \exp \Biggl( r( \log n + C') + (r-1)(-\log n + \log \log n + C') \\
&& \mbox{ \ \ \ \ \ \ \ \ \ \ \ \ \ \ } - r (\log n + d) +  Cr^2 \frac{\log n + c}{bn} \Biggr) \\
&\leq& \sum_{r=2}^{N'} \exp \left( \log n + 2r \log \log n - r \log n + \frac{1}{3}r \log n \right) = o(1).
\end{eqnarray*}
\end{proof}

Suppose $G$ contains two components of size $\geq 2$. From Lemma~\ref{lemma: connectivity critical}, we may assume that both have at least size $\beta n$ for some $\beta>0$. Considering Lemma~\ref{lemma: changing vertex type 2}, the probability that these components are not connected is \( \exp ( -O(n \log n)) = o(1). \)
Hence $G$ consists of a giant component and some isolated vertices w.h.p. To prove the second part of Theorem~\ref{main thm: critical connectivity} we use the ideas in \cite{Durrett2006}.

\begin{theorem}
\label{thm: poisson}
For $\lambda = (\log n + c + o(1))/b$, the number of isolated vertices $I_n$ in $G$ converges to Poi$(Ae^{-c})$ where \( A = \sum_{\{k: b_k = b\}} a_k. \) 
\end{theorem}

\begin{proof}
Following (\ref{eq: degree 0}), the probability that a type $k$ vertex is isolated is $q_k \sim e^{-cb_k/b}n^{-b_k / b}$. Reorder the vertex types such that $b_1 \leq b_2 \leq \dots \leq b_m$ and assume $b = b_1 = \dots = b_l \ (1 \leq l \leq m)$. Then
\[
\E I_n = \sum_{k=1}^{m} a_k n q_k \sim \sum_{k=1}^{l}a_k e^{-c} = Ae^{-c}.
\]
We now calculate the expected number of ordered 2-tuples of isolated vertices by
\begin{eqnarray*}
\E I_n (I_n -1) &=& \sum_{k=1}^{m} a_k n (a_k n -1) q_k^2 (1-p_{kk})^{-1}+ \sum_{k_1 \neq k_2} (a_{k_1}n)(a_{k_2}n)q_{k_1}q_{k_2}(1-p_{k_1 k_2})^{-1} \\
&\sim& A^2 e^{-2c}.
\end{eqnarray*}
Analogously, the expected number of ordered $r$-tuples of isolated vertices is $\sim A^re^{-rc}$, which implies that $I_n$ converges to Poi$(Ae^{-c})$.
\end{proof}

\begin{remark}
\cite{Durrett2006} shows that for $p = (\log n + c + o(1))/n$, the number of isolated vertices converge to Poi$(e^{-c})$. Thus we obtained a similar result for inhomogeneous random graphs, but the difference is that there is a $b$ in the denominator of $p$ and that there is a factor $A$ which is the sum of $a_k$ for $k$ satisfying $b_k = b$. The occurrence of $A$ is natural since the more the number of vertex types satisfying $b_k=b$, the more isolated vertices. Theorem~\ref{thm: poisson} also implies that the probability that $G$ is connected converges to $\exp(-Ae^{-c})$.
\end{remark}


\subsection{Connectivity for random graphs with multiple blocks}
Recall the second model defined in Definition~\ref{def: graph model 2}. The proof for Theorem~\ref{main thm: connectivity 2} readily follows from Theorem~\ref{main thm: connectivity 1}. If $\alpha b_k ^{(i)} < 1$ for some $i \in \{ 1, \dots , r\}$ and $k \in \{1, \dots , m_i \}$, following the proof of Theorem~\ref{main thm: connectivity 1}(i) gives the existence of sufficiently many isolated vertices of type $k$ in block $i$. Note that for a vertex $x$, $\prob[d_x=0]$ is barely affected by $q$ and $p'(n)$ that occurs from other blocks. If $\alpha b_k ^{(i)} >1$ for all $i$ and $k$, then each block $G_i$ is connected w.h.p.  Using the assumption in Remark~\ref{main thm: assumption 2}, the probability that $G_i$ and $G_j$ are not connected is $\sim \exp(-C_{ij} n^2 p'(n))$ for some $C_{ij}>0$ by Lemma~\ref{lemma: changing vertex type 2}. Hence we additionally need $n^2 p'(n) \to \infty$ for $G$ to be connected.

\paragraph{Acknowledgement}
This article was written as part of the Undergraduate Research Program at Seoul National University. I would like to express my gratitude to Professor Insuk Seo for presenting the research topic and guiding the discussions.

\printbibliography

@incollection{Alon1995,
	address = {New York, NY},
	author = {Alon, Noga},
	booktitle = {Discrete Probability and Algorithms},
	copyright = {Springer Science+Business Media New York 1995},
	pages = {11--14},
	publisher = {Springer New York},
	series = {The IMA Volumes in Mathematics and its Applications},
	title = {A Note on Network Reliability}}

@article{Gilbert1959,
	author = {E. N. Gilbert},
	doi = {10.1214/aoms/1177706098},
	journal = {The Annals of Mathematical Statistics},
	number = {4},
	pages = {1141 -- 1144},
	publisher = {Institute of Mathematical Statistics},
	title = {Random Graphs},
	url = {https://doi.org/10.1214/aoms/1177706098},
	volume = {30},
	year = {1959}}

@book{Durrett2006,
	author = {Durrett, Rick},
	collection = {Cambridge Series in Statistical and Probabilistic Mathematics},
	place = {Cambridge},
	publisher = {Cambridge University Press},
	series = {Cambridge Series in Statistical and Probabilistic Mathematics},
	title = {Random Graph Dynamics},
	year = {2006}}

@book{Harris1963,
	address = {Berlin},
	author = {Harris, Theodore E.},
	description = {The theory of branching processes},
	mrclass = {60.67},
	mrnumber = {MR0163361 (29 \#664)},
	mrreviewer = {P. A. P. Moran},
	pages = {xiv+230},
	publisher = {Springer-Verlag},
	series = {Die Grundlehren der Mathematischen Wissenschaften, Bd. 119},
	title = {The theory of branching processes},
	year = {1963}}

@book{Janson2000,
	address = {New York; Chichester},
	author = {Janson, Svante and Luczak, Tomasz and Rucinski, Andrzej},
	description = {Random Graphs: Svante Janson, Tomasz Luczak, Andrzej Rucinski: 9780471175414: Amazon.com: Books},
	isbn = {0471175412 9780471175414},
	publisher = {John Wiley \& Sons},
	refid = {43340250},
	title = {Theory of random graphs},
	url = {http://www.amazon.com/Random-Graphs-Svante-Janson/dp/0471175412},
	year = {2000}}

@inbook{Roch2024,
	author = {Roch, S{\'e}bastien},
	booktitle = {Modern Discrete Probability: An Essential Toolkit},
	collection = {Cambridge Series in Statistical and Probabilistic Mathematics},
	doi = {10.1017/9781009305129.007},
	pages = {327--396},
	place = {Cambridge},
	publisher = {Cambridge University Press},
	series = {Cambridge Series in Statistical and Probabilistic Mathematics},
	title = {Branching Processes},
	year = {2024}}

@book{Bollobas2001,
	author = {Bollob{\'a}s, B{\'e}la},
	collection = {Cambridge Studies in Advanced Mathematics},
	edition = {2},
	place = {Cambridge},
	publisher = {Cambridge University Press},
	series = {Cambridge Studies in Advanced Mathematics},
	title = {Random Graphs},
	year = {2001}}

@article{Erdos1959,
	author = {Erd\"os, P. and R\'enyi, A.},
	journal = {Publicationes Mathematicae Debrecen},
	pages = {290--297},
	title = {On Random Graphs I},
	volume = 6,
	year = 1959}

@article{Erdos1960,
	author = {Erd\"os, Paul and R\'enyi, Alfred},
	journal = {Publ. Math. Inst. Hungary. Acad. Sci.},
	pages = {17-61},
	title = {On the evolution of random graphs},
	volume = 5,
	year = 1960}

@article{Soderberg2003,
	author = {S{\"o}derberg, Bo},
	doi = {10.1103/PhysRevE.66.066121},
	journal = {Physical review. E, Statistical, nonlinear, and soft matter physics},
	month = {01},
	pages = {066121},
	title = {A General Formalism for Inhomogeneous Random Graphs},
	volume = {66},
	year = {2003}}

@article{Bollobas2007,
	address = {USA},
	author = {Bollob\'{a}s, B\'{e}la and Janson, Svante and Riordan, Oliver},
	issn = {1042-9832},
	issue_date = {August 2007},
	journal = {Random Struct. Algorithms},
	month = {aug},
	number = {1},
	numpages = {120},
	pages = {3--122},
	publisher = {John Wiley \& Sons, Inc.},
	title = {The phase transition in inhomogeneous random graphs},
	volume = {31},
	year = {2007}}

\end{document}